\documentclass[12pt,reqno]{amsart}

\usepackage[headings]{fullpage}
\usepackage{amssymb,amsmath,mathtools,bbm,tikz,tikz-cd,color,relsize,multirow}
\usepackage[all,cmtip]{xy}
\usepackage{url}
\usepackage{blkarray}
\usepackage{fancybox}
\usepackage{cancel}
\usetikzlibrary{positioning}
\usetikzlibrary{calc}
\usetikzlibrary{decorations.markings}
\usetikzlibrary{arrows.meta}
\tikzset{> /.tip = {Stealth[round,length=6pt]}}
\tikzstyle over=[preaction={draw,line width=5pt,white}]
\usetikzlibrary{calc}
\usetikzlibrary{decorations.markings}
\tikzstyle{hvector}=[inner sep=2pt,draw=blue!50,fill=blue!10,thick]
\tikzstyle{unit}=[inner sep=2pt,shape=circle, draw]
\tikzstyle{counit}=[inner sep=2pt,shape=circle, draw,fill=gray]
\tikzstyle{antipode}=[inner sep=2pt,shape=rectangle, draw]
\tikzstyle{cocycle}=[inner sep=2pt,shape=circle, draw]
\tikzstyle{twistedm}=[inner sep=2pt,shape=circle, fill=gray]
\tikzstyle{autom}=[inner sep=2pt,shape=circle, draw]
\tikzstyle{coact}=[inner sep=2pt,shape=circle, fill=black]

\allowdisplaybreaks

\usepackage[bookmarks=true,%
    colorlinks=true,%
    linkcolor=blue,%
    citecolor=blue,%
    filecolor=blue,%
    menucolor=blue,%
    urlcolor=blue,%
    breaklinks=true]{hyperref}
\usepackage{slashed}    
\usepackage{verbatim}
\usepackage[normalem]{ulem}  
\usepackage{diagbox}
\usepackage{mathabx}

\newtheorem{theorem}{Theorem}[section]
\theoremstyle{definition}
\newtheorem{proposition}[theorem]{Proposition}
\newtheorem{lemma}[theorem]{Lemma}

\newtheorem{remark}[theorem]{Remark}

\newtheorem{corollary}[theorem]{Corollary}

\newcommand{\autour}[1]{\tikz[baseline=(X.base)]\node
  [draw=black,fill=cyan!20,semithick,rectangle,inner sep=2pt,
  rounded corners=3pt] (X) {#1};}

\def\BZ{\mathbbm Z}
\def\BQ{\mathbbm Q}

\def\la{\langle}
\def\ra{\rangle}

\def\tq{\tilde{q}}

\def\a{\alpha}
\def\b{\beta}
\def\g{\gamma}

\def\ve{\varepsilon}

\def\be{\begin{equation}}
\def\ee{\end{equation}}

\def\RT{F}
\def\tr{\mathrm{tr}}
\def\diag{\mathrm{diag}}

\newcommand{\slthree}{\mathfrak{sl}_3}

\def\End{\mathrm{End}}
\def\hs{\overline{s}}
\def\LG{\mathrm{LG}}

\def\ADO{\mathrm{ADO}}

\begin{document}

\title[Skein theory for the Links--Gould polynomial]{
  Skein theory for the Links--Gould polynomial}
\author[S. Garoufalidis]{Stavros Garoufalidis}
\address{
  International Center for Mathematics, Department of Mathematics \\
  Southern University of Science and Technology \\
  Shenzhen, China \newline
  {\tt \url{http://people.mpim-bonn.mpg.de/stavros}}}
\email{stavros@mpim-bonn.mpg.de}

\author[M. Harper]{Matthew Harper}
\address{Michigan State University, East Lansing, Michigan, USA}
\email{mrhmath@proton.me}

\author[R. Kashaev]{Rinat Kashaev}
\address{Section de Math\'ematiques, Universit\'e de Gen\`eve \\
rue du Conseil-G\'en\'eral 7-9, 1205 Gen\`eve, Switzerland \newline
         {\tt \url{http://www.unige.ch/math/folks/kashaev}}}
\email{Rinat.Kashaev@unige.ch}
       
\author[B.-M. Kohli]{Ben-Michael Kohli}
\address{Section de Math\'ematiques, Universit\'e de Gen\`eve \\
rue du Conseil-G\'en\'eral 7-9, 1205 Gen\`eve, Switzerland}
\email{bm.kohli@protonmail.ch}

\author[J. Song]{Jiebo Song}
\address{Beijing Institute of Mathematical Sciences and Applications,
  Huairou District, Beijing, China}
\email{songjiebo@bimsa.cn}

\author[G. Tahar]{Guillaume Tahar}
\address{Beijing Institute of Mathematical Sciences and Applications,
  Huairou District, Beijing, China}
\email{guillaume.tahar@bimsa.cn}

\thanks{
  {\em Key words and phrases:}
  knots, links, Links--Gould invariant, ADO invariant, braid groups,
  cubic skein theory, $R$-matrices, quantum groups, Nichols algebras,
  Yetter--Drinfel'd modules, Seifert genus.\newline
  {\em MSC: primary 57K14, 57K16; secondary 17B37, 18M15, 20F36}
}

\date{16 February 2026}

\begin{abstract}
Building further on work of Marin and Wagner, we give a cubic braid-type skein theory of the Links--Gould polynomial invariant of oriented links and prove that it can be used to evaluate any oriented link, adding this polynomial to the list of polynomial invariants that can be computed by skein theory. As a consequence, we prove that this skein theory is also shared by the $V_1$-polynomial defined by two of the authors, deducing the equality of the two link polynomials. This implies specialization properties of the $V_1$-polynomial to the Alexander polynomial and to the $\ADO_3$-invariant, the fact that it is a Vassiliev power series invariant, as well as a Seifert genus bound for knots.
\end{abstract}

\maketitle

{\footnotesize
\tableofcontents
}


\section{Introduction}
\label{sec.intro}

\subsection{Link polynomials via skein theory}

The origin of skein theory as a way to uniquely define and compute invariants of
links goes back to at least Alexander in 1928 and Conway in
1969~\cite{Alexander,Conway} who used it to compute the Alexander polynomial $\Delta_L(t) \in \BZ[t^{\pm 1/2}]$ and the Conway polynomial $\nabla_L(z) \in \BZ[z]$
of an oriented link $L$ in $S^3$. These polynomials satisfy the skein relations
\be
\Delta_{L_+}(t) - \Delta_{L_-}(t) = (t^{1/2}-t^{-1/2})\Delta_{L_0}(t),
\qquad
\nabla_{L_+}(z) - \nabla_{L_-}(z) = z \nabla_{L_0}(z)
\ee
where, \( L_+, L_-, L_0 \) are three link diagrams that are identical except at a
single crossing, where they differ by being a positive crossing, a negative crossing,
or a smoothing, respectively. Such a relation, together with the fact that any
diagram of a link can be changed to a diagram of an unlink, and the initial condition
$\Delta_{\text{unknot}}(t)=1$ or $\nabla_{\text{unknot}}(z)=1$, computes
the polynomial on any link diagram. The two invariants are related by a change of
variables. 

But a dramatic revival of skein theory came with Jones's discovery of the Jones
polynomial $J_L(q) \in \BZ[q^{\pm 1/2}]$ which satisfies the skein-theory
\be
q^{-1} J_{L_+}(q) - q J_{L_-}(q) = (q^{1/2} - q^{-1/2}) J_{L_0}(q), \qquad
J_{\text{unknot}}(q)=1 \,.
\ee
Soon after Jones's discovery, it was realized by a group of six people that there
is a 2-variable generalization of the Jones and the Alexander/Conway polynomial,
the so-called HOMFLYPT-polynomial that satisfies the skein theory~\cite{HOMFLY, PT}
\be
l P_{L_+}(m,l) - l^{-1} P_{L_-}(m,l) = m P_{L_0}(m,l), \qquad
P_{\text{unknot}}(m,l) = 1 \,.
\ee
At this point, skein-theory and skein modules (on surfaces and 3-manifolds) were
introduced and studied independently by Przytycki~\cite{PT,PT:conway} and
Turaev~\cite{Turaev:skein}. At the same time the Kauffman polynomial was introduced, and it satisfies a quadratic skein relation 
\be
Q_{L_+}(z)+Q_{L_-}(z)=zQ_{L_0}(z)+zQ_{L_\infty}(z)
\ee
for framed unoriented links, where $L_\infty$ is the horizontal smoothing,
\cite{Kauff2}.

An excellent introduction of the history of skein theory and its role in knot theory
is given by Przytycki~\cite{Prz:skein}, where one may find examples of linear and
quadratic skein theory.

On the other hand, cubic skein theory is a more difficult subject, and in fact it was
only recently shown that the move that replaces three half-twists by
none is not sufficient to untie all knots~\cite{BCGIKMMPW}.

Our paper gives an example of a cubic skein theory which uniquely determines the
two-variable Links-Gould invariant~\cite{LG} and identifies it with the recent
$V_1$ two-variable polynomial of~\cite{GK:multi}. This identification will be used
in a subsequent work to provide Seifert genus bounds for the colored Links-Gould
polynomials and of the $V_n$-polynomials of~\cite{GK:multi} which are independent of $n$,
extending the work of~\cite{KT} and confirming some patterns found
in~\cite{GL:patterns}.

\subsection{Multivariable knot polynomials}

Recently, a systematic way to define and effectively compute multivariable knot
polynomials was introduced in~\cite{GK:multi}, using as input a finite dimensional
Nichols algebra (or a finite-dimensional Drinfeld--Yetter module of it)
with an automorphism. From such an algebra, one can define a rigid $R$-matrix and
construct a state-sum invariant of long knots by applying the well-known
Reshetikhin--Turaev functor.

Nichols algebras are easy to describe, and 
in the simplest case of rank 1, such an algebra is uniquely determined by the
data
\be
\text{basis}(V)=\{x\}, \qquad \Delta(x)=x \otimes 1 + 1 \otimes x,
\qquad \tau(x\otimes x)=q x\otimes x,
\qquad \phi(x)=t x \,.
\ee
In~\cite{GK:multi}, it was shown that the corresponding invariants are the ADO
invariants of a knot~\cite{ADO}, and the colored Jones polynomials of a knot
~\cite{RT,Tu:YB}.

The next case of a Nichols algebra of rank 2 leads to a family of 2-variable
polynomials $\Lambda_\omega(t_0,t_1)$ at each root of unity $\omega$ and a
sequence $V_n(t,q)$ of 2-variable polynomials, where $n \geq 1$ is an integer. 

The above polynomials are defined for oriented long knots, but the
construction can often be extended to polynomial invariants of framed, oriented
links in 3-space. Whereas a general theorem is not known for all finite dimensional
Nichols algebras with automorphisms, it was shown in~\cite{GHKST} that the
long knot $V_1$, $\Lambda_1$, and $\Lambda_{-1}$ polynomials do extend
to framed, oriented links (and that they are independent of the framing), and two
of them were identified with the Alexander polynomial and the $\slthree$-link polynomial
of~\cite{Harper}
\be
\Lambda_{1,L}(t_0, t_1) = \Delta_L (t_0) \, \Delta_L(t_1), \qquad
\Lambda_{-1,L}(t^{-2}, s^{-2}) = \Delta_{\slthree,L}(t,s) 
\ee
as was conjectured in~\cite{GK:multi}. In the following we will identify the $V_1$ polynomial
of~\cite{GK:multi} with the Links--Gould polynomial~\cite{LG}, as was conjectured in
~\cite{GK:multi}.

\subsection{\texorpdfstring{$\LG = V_1$}{LG=V1} via skein theory}

Throughout the paper, all links will be oriented and considered
up to ambient isotopy in 3-space. 

\begin{theorem}
\label{thm.1}
For all links $L$ we have:
\be
V_{1,L}(t_0,t_1) = \LG_L(t_0,t_1) \in \BZ[t_0^{\pm 1},t_1^{\pm 1}] \,.
\ee
\end{theorem}

Whereas both polynomial invariants $V_1$ and $\LG$ are defined by 4-dimensional
$R$-matrices, we were unable to show that these are conjugate or weakly-conjugate
(borrowing terminology from~\cite{GHKST}), and hence we could not use the methods
of~\cite{GHKST} to deduce the above theorem.

Instead, we prove the above theorem by
showing that both invariants satisfy a common skein theory that uniquely determines
them, hence equality follows. This common skein theory that we shortly discuss does
not describe a presentation of the braided monoidal category of representations of
$U_q(\mathfrak{sl}(2|1))$, but instead is tailored to relations on the braid-group
representations of these invariants, and ultimately to polynomial equations satisfied
by the $R$-matrices of both the $V_1$ and the the $\LG$ invariants.

To describe these rather complicated skein relations, we choose to present braids
algebraically rather than pictorially, as words in the standard generators $s_i$
and their inverses $\hs_i$ for $i=1,\dots,n-1$ of the Artin braid group $B_n$
~\cite{Artin} shown in Figure~\ref{fig:FIG1}. 

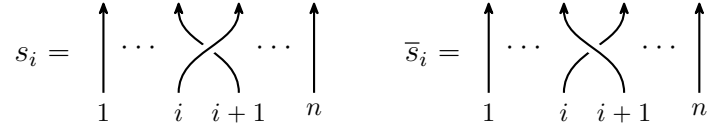
\begin{figure}[htpb!]
$s_i=$\,
\begin{tikzpicture}[baseline=12, xscale=.4, yscale=.4, ]
\coordinate (b) at (-2.5,0);
\coordinate (c) at (4.5,0);
%
\draw[thick,->] (b) -- ($(b)+(0,3)$);
\draw[thick,->] (c) -- ($(c)+(0,3)$);
%
 \node[below] at (b) {\footnotesize{$1$}};
 \node[below] at (c) {\footnotesize{$n$}};
%
\node at (-1.25,1.5) {$\cdots$};
\node at (3.25,1.5) {$\cdots$};
\node[below] at (0,0) {\footnotesize{$i$}};
\node[below] at (2,0) {\footnotesize{$i+1$}};
\draw[thick,->] (2,0) to [out=90, in=-90] (0,3);
\draw[over,thick,->] (0,0) to [out=90, in=-90] (2,3);
\end{tikzpicture} 
\qquad
$\hs_i=$~
\begin{tikzpicture}[baseline=12, xscale=.4, yscale=.4]
\coordinate (b) at (-2.5,0);
\coordinate (c) at (4.5,0);
%
\draw[thick,->] (b) -- ($(b)+(0,3)$);
\draw[thick,->] (c) -- ($(c)+(0,3)$);
%
 \node[below] at (b) {\footnotesize{$1$}};
 \node[below] at (c) {\footnotesize{$n$}};
%
\node at (-1.25,1.5) {$\cdots$};
\node at (3.25,1.5) {$\cdots$};
\node[below] at (0,0) {\footnotesize{$i$}};
\node[below] at (2,0) {\footnotesize{$i+1$}};
\draw[thick,->] (0,0) to [out=90, in=-90] (2,3);
\draw[over, thick,->] (2,0) to [out=90, in=-90] (0,3);
\end{tikzpicture} 
\caption{The standard generators $s_i$ of the braid
  group $B_n$ and their inverses $\hs_i$ for $i=1,\dots,n-1$.} 
\label{fig:FIG1}
\end{figure}

There is a natural inclusion of $B_n \to B_{n+1}$ obtained by adding a vertical
strand on the right, and as is customary in the literature (see e.g.,~\cite{Farb}
and references therein), we denote by $s_i$ the corresponding elements in $B_n$
and in $B_{n+1}$. 

\begin{lemma}
\label{lem.R123}  
Both $\LG$ and $V_1$ satisfy the skein relations in $B_n$ and
$i,j,k =1, \dots, n-1$: 
\be 
\label{R1}\tag{$R_1$}
s_i^2 + (1-t_0 - t_1) \, s_i + (t_0 t_1 - t_0 - t_1) \,1
+ (t_0 t_1) \, \hs_i = 0 \,,
\ee
\be
\label{R2}\tag{$R_2$}
\hs_i s_j s_i - s_i s_j \hs_i - \hs_i \hs_j s_i + s_i \hs_j \hs_i =
s_i s_j - s_i \hs_j - \hs_i s_j + \hs_i \hs_j - s_j s_i + s_j \hs_i
+ \hs_j s_i - \hs_j \hs_i 
\ee
 for $|i-j|=1$,
\be
\label{R3}
\tag{$R_3$}
s_i \hs_k s_j \hs_k - \hs_k s_j \hs_k s_i = \sum_{l=1}^{78} a_l w_l 
\ee
for $k-2=j-1=i$, where $w$ is 
\begin{small}
\be
\label{w78}
\begin{aligned}
  w = & (s_i s_j, s_i \hs_j, s_j s_i, \hs_j s_i, s_k s_j, \hs_k s_j, s_k \hs_j,
  \hs_k \hs_j, s_j s_k, \hs_j s_k, s_j \hs_k, \hs_j \hs_k, s_i s_k s_j,
  s_i s_k \hs_j, s_i \hs_k s_j, \\ & s_i \hs_k \hs_j, \hs_i s_k s_j, \hs_i s_k \hs_j,
  \hs_i \hs_k s_j, \hs_i \hs_k \hs_j , s_i s_j s_k, s_i s_j \hs_k,
  s_i \hs_j s_k , s_i \hs_j \hs_k, \hs_i s_j s_k, \hs_i s_j \hs_k,
  \hs_i \hs_j s_k, \\ & \hs_i \hs_j \hs_k, s_j s_i s_k, s_j s_i \hs_k, s_j \hs_i s_k,
  s_j \hs_i \hs_k, \hs_j s_i s_k, \hs_j s_i \hs_k, \hs_j \hs_i s_k,
  \hs_j \hs_i \hs_k, s_k s_j s_i, s_k s_j \hs_i, s_k \hs_j s_i, \\ &
  s_k \hs_j \hs_i, \hs_k s_j s_i, \hs_k s_j \hs_i, \hs_k \hs_j s_i,
  \hs_k \hs_j \hs_i, s_k \hs_i \hs_j \hs_i, \hs_k \hs_i \hs_j \hs_i,
  s_k \hs_i \hs_j s_i, \hs_k \hs_i \hs_j s_i, s_k \hs_i s_j \hs_i,
  \hs_k \hs_i s_j \hs_i, \\ & s_k s_i \hs_j \hs_i, \hs_k s_i \hs_j \hs_i, 
  s_k s_i \hs_j s_i, \hs_k s_i \hs_j s_i, s_k s_i s_j \hs_i,
  \hs_k s_i s_j \hs_i, \hs_i \hs_j \hs_i s_k, \hs_i \hs_j \hs_i \hs_k, 
  \hs_i \hs_j s_i s_k,  \hs_i \hs_j s_i \hs_k, \\ & \hs_i s_j \hs_i s_k, 
  \hs_i s_j \hs_i \hs_k, s_i \hs_j \hs_i s_k, s_i \hs_j \hs_i \hs_k,
  s_i \hs_j s_i s_k, s_i \hs_j s_i \hs_k, s_i s_j \hs_i s_k,
  s_i s_j \hs_i \hs_k, s_i \hs_j \hs_k s_j, s_i \hs_j s_k \hs_j, \\ &
  s_i s_j \hs_k \hs_j, s_i s_j \hs_k s_j, s_i s_j s_k \hs_j,
  \hs_j \hs_k s_j s_i, \hs_j s_k \hs_j s_i, s_j \hs_k \hs_j s_i,
  s_j \hs_k s_j s_i, s_j s_k \hs_j s_i) 
\end{aligned}
\ee
\end{small}
and the coefficients
$a_l \in \BQ(t_0,t_1)$ for $l=1,\ldots,78$ are given explicitly
in Appendix~\ref{sec.R123}.
\end{lemma}

The relation \eqref{R1} is derived from the $R$-matrix $R_{\LG}$
given in Appendix~\ref{sec.Rmat} whose minimal polynomial is cubic with
roots $1,t_0,t_1$. The relation \eqref{R2} was discovered by Ishii \cite{Ish2}.
The existence of a relation \eqref{R3} was proven by Marin--Wagner 
\cite[Sec.6.2, Sec.6.3]{MW} with no explicit description. Since the support of
~\eqref{R3} is important in the reduction algorithm of Theorem~\ref{thm.2}
below, we give its coefficients explicitly, and explain in Section~\eqref{sub.how}
how it was found.

Note that \eqref{R1}, \eqref{R2}
and \eqref{R3} come from relations in the braid groups $B_2$,
$B_3$, and $B_4$ involving 2, 3, and 4 braid strands respectively. 

The proof of the above lemma follows from the fact that the $R$-matrices
for $V_1$ and $\LG$ given in the appendix satisfy the polynomial identities
\eqref{R1}, \eqref{R2} and \eqref{R3}, a fact certified by a computer calculation. 

An important complement of the above skein relations is their completeness, that is
they allow the computation of the invariant for every link. This is achieved
by an effective reduction algorithm given in Theorem~\ref{thm.2} below.
To phrase it, consider the quotient 
\be
\label{Cndef}
C_n=\BQ(t_0,t_1)[B_{n}]/(R_1,R_2,R_3)
\ee
of the group-algebra $\BQ(t_0,t_1)[B_n]$ of the braid group $B_n$ by the
2-sided ideal $(R_1,R_2,R_3)$. $C_n$ is an associative, non-commutative unital
algebra over the field $\BQ(t_0,t_1)$.

Note that if $c \in C_m$ for $m <n$ and $\beta \in B_n$, then
$\beta c, c \beta \in C_n$. In particular, there is a natural map $C_m \to C_n$
obtained from the braid group inclusion $B_m \to B_n$. 

\begin{theorem}
\label{thm.2}
For every $n \geq 3$, there is a reduction algorithm that implies
an equality 
\be
\label{CC}
C_n = C_{n-1} + C_{n-1} s_{n-1} C_{n-1} + C_{n-1} \hs_{n-1} C_{n-1} +
C_{n-2} \hs_{n-1} s_{n-2} \hs_{n-1} 
\ee
of $\BQ(t_0,t_1)$-vector spaces. 
\end{theorem}

In other words, the theorem above asserts that for $n \geq 3$, every braid
$\beta \in B_n$ can be reduced to an element of the right hand side of
Equation~\eqref{CC}.

Theorem~\ref{thm.2} is an effective version of Theorem 5.4 (ii) and Theorem 6.1
(after fixing a typo) of Marin--Wagner~\cite{MW}.

The next remark is important for specialization of this skein theory, e.g., to
the case of $\ADO_\omega$.

\begin{remark}
\label{rem.denom}
Although $C_n$ is a $\BQ(t_0,t_1)$-algebra and the relations \eqref{R1}, \eqref{R2}
and \eqref{R3}
as well as the proof of Theorem~\ref{thm.2} involve denominators, the statement
and the proof are valid if we replace the field $\BQ(t_0,t_1)$ with the ring
\be
\BZ[t_0,t_1, \delta(t_0,t_1)^{-1}]
\ee
where
\be
\delta(t_0,t_1) = t_0 t_1 (t_0 + t_1)(t_0 t_1 +1)(t_0 t_1 -1)(1+t_0)(1+t_1)
(t_0 + t_1 -1) (1 + t_0 t_1 + t_0^2 t_1 + t_0 t_1^2) \,.
\ee
\end{remark}

The next remark concerns the dimension of $C_n$.

\begin{remark}
\label{rem.dimCn}  
A corollary of Theorem~\ref{thm.2} is that $C_n$ is a finite-dimensional
$\BQ(t_0,t_1)$-vector space. In fact, it is conjectured in~\cite{MW} and
in analogous algebras studied in~\cite{Anghel}, that
\be
\dim(C_n) = \frac{(2 n-2)!(2 n-1)!}{((n-1)!n!)^2}
\ee
with the first few values for $n=2,\dots,9$ given by
\be
3, 20, 175, 1764, 19404, 226512, 2760615, 34763300 \,.
\ee
Unfortunately, Equation~\eqref{CC} for $n=3$ gives only the bound
$\dim(C_3) \leq 22$, and in this case it can be improved to an explicit
spanning set of 20 elements which is linearly independent, hence deducing
$\dim(C_3)=20$; see Corollary~\ref{basisC3} below. But beyond that, although
Theorem~\ref{thm.2} constructs explicit spanning sets for $C_n$, it does not give
sharp bounds for $\dim(C_n)$ for $n>3$.
\end{remark}

A straightforward consequence of Theorem~\ref{thm.2} is the following.

\begin{corollary}
\label{cor.1}  
A link invariant that satisfies the skein relations \eqref{R1}, \eqref{R2}
and \eqref{R3} and vanishes on split links is uniquely
determined by its value on the unknot.
\end{corollary}

This corollary combined with Lemma~\ref{lem.R123} and Lemma~\ref{lem.split}
below implies Theorem~\ref{thm.1}.

Corollary~\ref{cor.1} has an alternative formulation that swaps the global condition
of vanishing on diagrams of split links for an extra local skein relation that
involves tangles (as opposed to braids). Consider the relation ~\eqref{S2}
introduced by Ishii \cite{Ish3}:
\be
\label{S2}\tag{$S_2$}
\begin{tikzpicture}[scale=.7,baseline=15]
\draw[thick,->] (.25,1) to[out=90,in=-135] (1,2);
\draw[over, thick] (0,2) to[out=-45,in=90] (.75,1);
\draw[thick,->] (.75,1) to[out=-90,in=45] (0,0);
\draw[over, thick] (.25,1) to[out=-90,in=135] (1,0);
\end{tikzpicture}
-(t_0t_1+1)
\begin{tikzpicture}[scale=.7,baseline=15]
\draw[thick,->] (0,2) to[out=-60,in=90] (.25,1)
to[out=-90,in=60] (0,0);
\draw[thick,->] (1,0) to[out=120,in=-90] (.75,1) to[out=90,in=-120] (1,2);
\end{tikzpicture}
+t_0t_1
\begin{tikzpicture}[scale=.7,baseline=15]
\draw[thick] (0,2) to[out=-45,in=90] (.75,1);
\draw[over,thick,->] (.25,1) to[out=90,in=-135] (1,2);
\draw[thick] (.25,1) to[out=-90,in=135] (1,0);
\draw[over,thick,->] (.75,1) to[out=-90,in=45] (0,0);
\end{tikzpicture}
+
2(t_0-1)(t_1-1)
\begin{tikzpicture}[scale=.7,baseline=15]
\draw[thick,->] (0,1.75) to[out=-60,in=180] (.5,1.25)
to[out=0,in=240] (1,1.75);
\draw[thick,<-] (0,.25) to[out=60,in=180] (.5,.75) to[out=0,in=120] (1,.25);
\end{tikzpicture}
=0\,.
\ee

In \cite[Prop.3.3]{Ish3}, Ishii shows that a link invariant that satisfies
\eqref{R1} and \eqref{S2} vanishes on split links. Therefore Corollary
\ref{cor.1} implies the following result.

\begin{corollary}
\label{carac2}
A link invariant that satisfies the skein relations \eqref{R1}, \eqref{R2},
$\eqref{S2}$ and \eqref{R3} is uniquely determined by its value on the unknot.
\end{corollary}

Another consequence of Corollary~\ref{cor.1} is that a rank 1 Nichols
algebra invariant, namely $\ADO_\omega$, is equal to a specialization of a rank 2
Nichols algebra invariant, namely the $\LG$ invariant, as conjectured by Geer and
Patureau-Mirand \cite[Conj.4.7]{GP} and by~\cite{GK:multi}. 

\begin{theorem}
\label{thm.ADO}
For every link $L$ we have 
\be
\ADO_{\omega,L}(t) = \LG_L(t^2,\omega^2 t^{-2}),
\ee
where $\omega = e^{2 \pi i/6}$.
\end{theorem}

This follows from the fact that the $R$-matrix for $\ADO_\omega$ satisfies
the $(t_0,t_1)=(t^2,\omega^2 t^{-2})$ specialization of \eqref{R1}, \eqref{R2}
and \eqref{R3}
and Remark~\ref{rem.denom}.

A partial case of the above theorem for links that come from closures of
5-strand braids was given by Takenov~\cite{Ta}.

Note that the multi-color version of the Geer--Patureau-Mirand conjecture remains
open.

Interestingly, Theorem~\ref{thm.ADO} gives an example of two $R$ matrices on
a 3 and a 4-dimensional vector space with the same knot polynomial invariant.
Any connection between these two $R$-matrices remains to be investigated.

\begin{remark}
The effective proof of Theorem~\ref{thm.2}, which leads to an effective computation
of the $V_1=\LG$-polynomials, is by no means comparable in speed to the tangle
computation of these invariants given in~\cite{GL:patterns}. Indeed, skein theory
computations have apparent exponential complexity whereas tangle computations
tend to have polynomial complexity. 
\end{remark}

\begin{remark}
    A different strategy to prove that $V_1=\LG$ would have been to follow the ideas developed by Geer and Patureau-Mirand in~\cite{GP2}, where they provide a complete set of skein-theoretic relations for the (multivariable) $\LG$ invariant by working in the ribbon category of typical modules of $U_q(\mathfrak{sl}(2|1))$. From this, they derive an algorithm for evaluating the multivariable $\LG$ invariant of links using oriented trivalent graphs colored by generic typical modules. However, bridging the gap between the combinatorial construction of $V_1$ and a formulation of these trivalent graphs without the representation-theoretic background appeared more complicated than simply extending $V_1$ to links, which is the approach we adopted in~\cite{GHKST}.
\end{remark}
  
\subsection{Specialization, Vassiliev invariants and genus bounds for \texorpdfstring{$V_1$}{V1}}

We now discuss some applications of our main Theorem~\ref{thm.1}.

Using the variables $(t_0,t_1)$, the equality of $\LG$ and $V_1$ and previously
known results for $\LG$ \cite{DWIL, Ish, Ko, KoPat} imply the following
corollaries conjectured in~\cite{GK:multi}.

\begin{corollary}
The $V_1$ polynomial of a link $L$ satisfies the specializations
\be
\label{special}
V_{1,L}(t_0 , t_0^{-1}) = \Delta_L(t_0)^2,
\qquad V_{1,L}(t_0 , -t_0^{-1}) = \Delta_L(t_0^2),
\qquad V_{1,L}(t_0,1) = V_{1,L}(1,t_1) = 1
\ee
\end{corollary}

The genus bounds for $\LG$ from~\cite{KT} imply the following.

\begin{corollary}
For a knot $K$, we have the bound 
\be
\mathrm{deg}_t V_{1,K}(t,q) \leq 4 \, \text{genus}(K) \,,
\ee
where $\deg_t$ of a Laurent polynomial in $t$ is the difference between the highest and the lowest power of $t$ and $\text{genus}(K)$ is the minimal genus
of an embedded oriented surface spanning $K$. 
\end{corollary}

Vassiliev power series invariants were introduced in~\cite{BG:MMR}.
Geer~\cite{Geer:vas} showed that the Links--Gould polynomial is obtained from
the evaluation of the Kontsevich integral under the $\mathfrak{sl}(2|1)$ weight-system. 

\begin{corollary}
\label{cor.vassiliev}  
The oriented link polynomial $V_1$ is a Vassiliev power series invariant.
\end{corollary}

\begin{remark}
\label{rem.1} 
Note that there are three sets of variables used in the literature, namely
$(t_0,t_1)$ introduced by Ishii~\cite{Ish}, $(q^\a,q)$ used in the
context of representation theory e.g. to study $\LG$~\cite{LG, KT} and $(t,\tq)$
used in~\cite{GK:multi}. This is a point that leads to much confusion. The
relations between these different sets of variables are
\be
\label{variables}
(t_0, t_1)=(t \tq^{-n/2}, t^{-1}\tq^{-n/2}),
\qquad (q^\a,q) = (t^{-1/2} \tq^{1/4}, \tq^{-1/2}) \,. 
\ee
\end{remark}  

In this work we mostly use the known properties of $\LG$ to deduce similar
results for the $V_1$-polynomial, but we can also do the converse. A corollary of
Theorem~\ref{thm.1} is the following nontrivial symmetry of the
$\LG$-polynomial due to Ishii. 

\begin{corollary}[{\cite[Thm.1]{Ish}}]
For any link $L$, we have
\be
\LG_L(t_0,t_1) \in \BZ[t_0^{\pm 1},t_1^{\pm 1}] \,.
\ee
\end{corollary}

We end this section with a caution regarding deducing statements by specialization
of skein-theory.

\begin{remark}
\label{rem.noskein}
Although the Alexander polynomial comes from an enhanced $R$-matrix, the skein
theory approach that proves Theorem~\ref{thm.1} cannot be used to provide
novel proofs of either of the known specializations
\be
\LG_L(t_0,t_0^{-1})=\Delta_L(t_0)^2, \qquad \LG_L(t_0,-t_0^{-1})=\Delta_L(t_0^2)
\ee
of \cite{DWIL, Ish, Ko, KoPat}. This is because the left side of the \eqref{R3} relation
vanishes under these specializations and thus, the reduction algorithm does
not apply. However, the symmetry of the relations \eqref{R1}, \eqref{R2},
\eqref{R3} in $t_0$ and $t_1$ gives an alternative proof of the fact 
\be
\begin{aligned}
\LG_L(t_0,t_1)=\LG_L(t_1,t_0)\,.
\end{aligned}
\ee
\end{remark}

\subsection{Organization of the paper}

In Section~\ref{sec.skein} we introduce the main properties of skein
relations based on the braid group and use them to prove Theorem~\ref{thm.2}. In
Section~\ref{sec.RT} we recall briefly the definition and common properties
of the three link polynomials $\LG$, $V_1$ and $\ADO_\omega$ that are the focus
of our paper. We also explain how we found the explicit relation \eqref{R3}.

In Appendix~\ref{sec.Rmat} we give the $R$-matrices of the three polynomial
invariants that we study, and in Appendix ~\ref{sec.R123} we give the
lengthy coefficients of the skein relation \eqref{R3}. Appendices
~\ref{proofthreestrand} and~\ref{prooffourstrand} are dedicated to the proof of two
technical results used in the proof of Theorem~\ref{thm.2}.

\subsection*{Acknowledgements}

The authors would like to thank Ivan Marin, J\'{o}zef Przytycki and Emmanuel Wagner
for enlightening conversations, as well as the anonymous referee for the valuable remarks. MH was partially supported through the NSF-RTG grant
\#DMS-2135960. BMK was partially supported through the BJNSF grant \#IS24066.


\section{Skein theory for the \texorpdfstring{$\LG$}{LG} polynomial}
\label{sec.skein}

In this section we prove Theorem~\ref{thm.2} with a reduction algorithm.
Recall that the braid group $B_n$ has standard generators 
$s_i$ with inverses $\hs_i$ for $i=1,\dots,n-1$ shown in Figure~\ref{fig:FIG1}.
They satisfy the relations
\be
s_i s_j s_i = s_j s_i s_j\,\,\mbox{for } |i-j|=1,
\qquad s_i s_j = s_j s_i \,\,\mbox{for } |i-j|>1 \,.
\ee

Fix a braid $\beta \in B_n$ for $n \geq 2$. We will prove that it can be reduced
to belong to the $\BQ(t_0,t_1)$-vector space on the right hand side of Equation~\eqref{CC}.

Note that if $n=2$, it follows from \eqref{R1} that every braid in $B_2$ can
be reduced to a linear combination of $1, s_1, \hs_1$.

We will first prove Theorem~\ref{thm.2} for $n = 3, 4$, which are the hardest cases.
Then we prove the result by induction on $n \geq 5$.

\subsection{The \texorpdfstring{$n=3$}{n = 3} case}

Let $\beta \in B_3$. By applying \eqref{R1}, reduce it to the case 
\be
\beta = s_1^{\varepsilon_1} s_2^{\varepsilon_2} s_1^{\varepsilon_3}
\ldots \text{ or } \beta = s_2^{\varepsilon_1} s_1^{\varepsilon_2}
s_2^{\varepsilon_3} \ldots \text{ , } \varepsilon_i = \pm 1.
\ee

\begin{lemma}[Equivalent formulation of \eqref{R2}]\label{equivR2}
The following relation is equivalent to relation \eqref{R2} modulo relation
\eqref{R1}. For $1\leq i\leq n-2$ and $j=i+1$:
\be\label{modR2}
\begin{aligned}
s_j \hs_i s_j& = (t_0 t_1 +1 -t_0 -t_1) \, s_i + (t_0 + t_1 - t_0 t_1 -1) \, \hs_i \\
& \phantom{sii}- (t_0 + t_1 -1) \, s_i s_j + (t_0 + t_1 -1) \,
\hs_i s_j + t_0 t_1 \, s_i \hs_j - t_0 t_1 \, \hs_i \hs_j \\
&\phantom{sii} + s_j s_i - s_j \hs_i - (t_0 + t_1 - t_0 t_1) \,
\hs_j s_i + (t_0 + t_1 - t_0 t_1) \, \hs_j \hs_i \\
&\phantom{sii} - (t_0 + t_1) \, s_i s_j \hs_i + (t_0 + t_1) \,
s_i \hs_j \hs_i + s_i s_j s_i - t_0 t_1 \, \hs_i \hs_j \hs_i 
+ t_0 t_1 \, \hs_j s_i \hs_j \,.
\end{aligned}
\ee
\end{lemma}

\begin{proof}
We start by writing $\eqref{R2} \cdot s_j$ : 
\begin{align*}
  (s_j \hs_i) s_j&  =   (\hs_i s_j s_i) s_j - (s_i s_j \hs_i)s_j
  - (\hs_i \hs_j s_i) s_j + (s_i \hs_j \hs_i) s_j - (s_i s_j)s_j + (s_i \hs_j)s_j \\
  &\phantom{sii} + (\hs_i s_j)s_j
  - (\hs_i \hs_j)s_j + (s_j s_i)s_j - (\hs_j s_i)s_j + (\hs_j \hs_i)s_j\\
   &  = \hs_i (s_i s_j s_i) - \hs_j s_i s_j s_j - s_j \hs_i \hs_j s_j
   + s_i s_i \hs_j \hs_i - s_i((t_0 + t_1 -1) \, s_j \\
   &\phantom{sii}  +(t_0 + t_1 - t_0 t_1) \,
   1 - t_0 t_1 \, \hs_j) + s_i \\
   &\phantom{sii} + \hs_i((t_0 + t_1 -1) \, s_j +(t_0 + t_1 - t_0 t_1) \,
   1 - t_0 t_1 \, \hs_j) - \hs_i + s_i s_j s_i - s_i s_j \hs_i + s_i \hs_j \hs_i\\
    &   =   s_j s_i - \hs_j s_i ((t_0 + t_1 -1) \, s_j +(t_0 + t_1 - t_0 t_1)
    \, 1 - t_0 t_1 \, \hs_j) - s_j \hs_i \\
   &\phantom{sii} + ((t_0 + t_1 -1) \, s_i +(t_0 + t_1 - t_0 t_1)
   \, 1 - t_0 t_1 \, \hs_i) \hs_j \hs_i \\
   &\phantom{sii} - (t_0 + t_1 -1) \, s_i s_j - (t_0 + t_1 - t_0 t_1)
   \, s_i + t_0 t_1 \, s_i \hs_j + s_i\\
   &\phantom{sii}  + (t_0 + t_1 -1) \, \hs_i s_j + (t_0 + t_1 - t_0 t_1)
   \, \hs_i - t_0 t_1 \, \hs_i \hs_j - \hs_i + s_i s_j s_i - s_i s_j \hs_i
   + s_i \hs_j \hs_i  \\
    &   =  s_j s_i - (t_0 + t_1 -1) \, s_i s_j \hs_i - (t_0 + t_1 - t_0 t_1)
    \, \hs_j s_i + t_0 t_1 \, \hs_j s_i \hs_j\\
    &\phantom{sii} - s_j \hs_i + (t_0 + t_1 -1) \, s_i \hs_j \hs_i
    + (t_0 + t_1 - t_0 t_1)
   \, \hs_j \hs_i - t_0 t_1 \, \hs_i \hs_j \hs_i\\
   &\phantom{sii} - (t_0 + t_1 -1) \, s_i s_j - (t_0 + t_1 - t_0 t_1)
   \, s_i + t_0 t_1 \, s_i \hs_j + s_i\\
   &\phantom{sii} + (t_0 + t_1 -1) \, \hs_i s_j + (t_0 + t_1 - t_0 t_1)
   \, \hs_i - t_0 t_1 \, \hs_i \hs_j - \hs_i + s_i s_j s_i - s_i s_j \hs_i
   + s_i \hs_j \hs_i\\
   &   =  (t_0 t_1 +1 -t_0 -t_1) \, s_i + (t_0 + t_1 - t_0 t_1 -1) \,
   \hs_i - (t_0 + t_1 -1) \, s_i s_j \\
   &\phantom{sii} + (t_0 + t_1 -1)
   \, \hs_i s_j + t_0 t_1 \, s_i \hs_j - t_0 t_1 \, \hs_i \hs_j
   + s_j s_i - s_j \hs_i - (t_0 + t_1 - t_0 t_1)
   \, \hs_j s_i\\
   &\phantom{sii}  + (t_0 + t_1 - t_0 t_1) \, \hs_j \hs_i
   - (t_0 + t_1) \, s_i s_j \hs_i + (t_0 + t_1)
   \, s_i \hs_j \hs_i + s_i s_j s_i - t_0 t_1 \, \hs_i \hs_j \hs_i\\
   &\phantom{sii} + t_0 t_1 \, \hs_j s_i \hs_j \,.
\end{align*}
\end{proof}

We now consider the inner automorphism $b \mapsto \check{b}$ of the braid group
$B_n$ defined by
\be
\label{innerB}
\check{b_n} = g_n^{-1} b g_n
\ee
where $g_n = (s_1 s_2 \ldots s_{n-1}) \ldots (s_1 s_2 s_3) (s_1 s_2) (s_1)$ is a
half twist in $n$ strands. Topologically, $\check{b}$ is braid $b$
``looked at from behind''.  It is easy to show the following. 

\begin{lemma}
\label{innerautomorphism}
The following hold true:
\begin{itemize}
\item
  for all $k=1,\dots,n-1$, we have $\widecheck{s_k} = s_{n-k}$,
\item
  for all $b, c \in B_n$, $\widecheck{b c} = \check{b} \check{c},$
\item
  for all $b \in B_n$, braids $b$ and $ \check{b}$ have the same link closure.
\end{itemize}
\end{lemma}

This automorphism will be used in the proof of Lemma~\ref{relations} below as well as in numerous locations in the Appendix~\ref{proofthreestrand} and ~\ref{prooffourstrand}. In addition, applying this automorphism shows that \eqref{modR2} of Lemma \ref{equivR2} is also true if you exchange the roles of $i$ and $j$, even though they do not play symmetric roles in the relation.

\begin{lemma}
\label{relations}
The following relations hold in $C_n$. In each $l$-letter relation, we assume
$1\leq i\leq n-l$, $j=i+1$, and $k=i+2$.\\
{
\noindent\underline{1-letter relations}
\begin{itemize}
  \item Inverse relation:
  \be
  \begin{aligned}
s_i \hs_i = \hs_i s_i = 1 \,,
\end{aligned}
\ee 
  \item Relation \eqref{R1} and its equivalent version: 
  \be
  \begin{aligned}
s_i^2 &= (t_0 + t_1 - 1) \, s_i + (t_0 + t_1 - t_0 t_1) \,1
- (t_0 t_1) \, \hs_i \,, \\
\hs_i^2 &= (t_0^{-1} + t_1^{-1} - 1) \, \hs_i + (t_0^{-1} + t_1^{-1}
- t_0^{-1} t_1^{-1}) \,1
- (t_0^{-1} t_1^{-1}) \, s_i\,.
\end{aligned}
\ee
\end{itemize}
\bigskip
\underline{2-letter relations}
\begin{itemize}
\item
  Far commutativity:
    \be\begin{aligned}s_l^{\pm 1} s_m^{\pm 1} = s_m^{\pm 1} s_l^{\pm 1} \,
      \text{ for $\vert l-m \vert \geq 2$}\,, 
\end{aligned}
\ee
\item
  The braid relation and its equivalent formulations:
  \be
  \begin{aligned}
      s_i s_j s_i &= s_j s_i s_j\,,
  \\
  (s_j^a s_i^a) s_j^b = s_i^b (s_j^a s_i^a) &\text{ and }
  (s_i^a s_j^a) s_i^b = s_j^b (s_i^a s_j^a) &
  \mbox{for $a,b=\pm1$,}
  \end{aligned}
  \ee
for $|i-j|=1$. 
\item
  Relation \eqref{R2} and its equivalent version:
\be
{\begin{aligned}
    \hs_i s_j &s_i - s_i s_j \hs_i - \hs_i \hs_j s_i + s_i \hs_j \hs_i =\\
    &s_i s_j - s_i \hs_j - \hs_i s_j + \hs_i \hs_j - s_j s_i + s_j \hs_i
    + \hs_j s_i - \hs_j \hs_i \,,
\end{aligned}}\ee\be
{\begin{aligned}
    \hs_j s_i &s_j - s_j s_i \hs_j - \hs_j \hs_i s_j + s_j \hs_i \hs_j =\\
    &s_j s_i - s_j \hs_i - \hs_j s_i + \hs_j \hs_i - s_i s_j + s_i \hs_j
    + \hs_i s_j - \hs_i \hs_j\,.
\end{aligned}}
\ee
\end{itemize}
\underline{3-letter relations}
\begin{itemize}
\item Relations implied by \eqref{R3}:
\be
\begin{aligned}
&s_i^{\pm 1} \hs_k s_j \hs_k = \hs_k s_j \hs_k s_i^{\pm 1} + \alpha \,, 
& 
&s_i^{\pm 1} s_k \hs_j  s_k = s_k \hs_j s_k s_i^{\pm 1} + \gamma \,,
\\
&s_k^{\pm 1} \hs_i s_j \hs_i = \hs_i s_j \hs_i s_k^{\pm 1} + \delta \,, 
&
&s_k^{\pm 1} s_i \hs_j s_i = s_i \hs_j s_i s_k^{\pm 1} + \eta \,,
\end{aligned}
\ee
\end{itemize}
where $\alpha$, $\gamma$ are $\BQ(t_0,t_1)$-linear combinations of braid words
with at most 4 letters and at most one $s_k^{\pm 1}$, and $\delta$, $\eta$ are
$\BQ(t_0,t_1)$-linear combinations of braid words with at most 4 letters and at
most one $s_i^{\pm 1}$.}
\end{lemma}

\begin{proof}
{
The only non-obvious relations are the 3-letter relations. The first of these
four relations is \eqref{R3} in the case of the positive exponents. In the case
of the negative exponents, the first relation is obtained from \eqref{R3} by
writing $s_i^{-1} \cdot \eqref{R3} \cdot s_i^{-1}$.
The two versions of the second relation follow from the two versions of the first
relation and from Lemma~\ref{equivR2}:
\be
s_i^{\pm 1}(s_k \hs_j s_k) = t_0 t_1 \, s_i^{\pm 1} (\hs_k s_j \hs_k) \, + \,
\text{words with at most 4 letters and at most one } s_k^{\pm 1} \,,
\ee
and
\be
(s_k \hs_j s_k)s_i^{\pm 1} = t_0 t_1 \,  (\hs_k s_j \hs_k)s_i^{\pm 1} \,
+ \, \text{words with at most 4 letters and at most one } s_k^{\pm 1} \,.
\ee
Replacing $s_i^{\pm 1} (\hs_k s_j \hs_k)$ and $(\hs_k s_j \hs_k)s_i^{\pm 1}$
in the different versions of \eqref{R3} using these two identities, we get both
versions of the second 3 letter identity of Lemma~\ref{relations}. Finally, the
third and fourth 3-letter relations are obtained from the first two relations by
applying the inner automorphism defined through Equation~\eqref{innerB}.}
\end{proof}

\begin{lemma}
\label{threestrand}
In $C_3$, any word $\beta \in B_3$ is a linear
combination of at most 3 letter words, each of which: 
\begin{itemize}
\item
  either has at most one $s_2^{\pm 1}$,
\item
  or is $\hs_2 s_1 \hs_2$.
\end{itemize}
\end{lemma}

We prove Lemma~\ref{threestrand} in Appendix~\ref{proofthreestrand}. This implies
Theorem~\ref{thm.2} in the three strand case $n=3$. Actually it implies an enhancement
of Equation~\eqref{CC} for $n=3$ given in the next corollary.

\begin{corollary}
\label{basisC3}
The following set is a $\BQ(t_0,t_1)$-linear basis of $C_3$:
\be
\label{bC3}
\begin{aligned}
\{1, s_1, \hs_1, s_2, \hs_2, s_2 s_1, s_2 \hs_1, \hs_2 s_1, \hs_2 \hs_1, s_1 s_2,
\hs_1 s_2, s_1 \hs_2, \hs_1 \hs_2, s_1 s_2 s_1, s_1 s_2 \hs_1, & \\
& \hspace{-10cm} s_1 \hs_2 \hs_1, \hs_1 \hs_2 s_1, \hs_1 \hs_2 \hs_1, \hs_1 s_2 \hs_1
\text{ (or $s_1 \hs_2 s_1$)}, \hs_2 s_1 \hs_2 \} \,,
\end{aligned}
\ee
hence the dimension of $C_3$ is $20$.
\end{corollary}

\begin{proof}
Using Lemma~\ref{threestrand}, the following set is a $\BQ(t_0,t_1)$-generating
set of $C_3$:
\be
\label{bC3b}
\begin{aligned}
\{1, s_1, \hs_1, s_2, \hs_2, s_2 s_1, s_2 \hs_1, \hs_2 s_1, \hs_2 \hs_1, s_1 s_2,
\hs_1 s_2, s_1 \hs_2, \hs_1 \hs_2, s_1 s_2 s_1, s_1 s_2 \hs_1, & \\
& \hspace{-10cm}
\hs_1 s_2 s_1,\hs_1 s_2 \hs_1, s_1 \hs_2 s_1, s_1 \hs_2 \hs_1, \hs_1 \hs_2 s_1,
\hs_1 \hs_2 \hs_1, \hs_2 s_1 \hs_2 \} \,.
\end{aligned}
\ee
Relation \eqref{R2} shows that the vector $\hs_1 s_2 s_1$ can be expressed
in terms of the other vectors of the family. Once that vector is removed,
\eqref{modR2} shows that
$s_1 \hs_2 s_1$ (or $\hs_1 s_2 \hs_1$) can also be removed. Thus we have
a generating set for $C_3$ with 20 elements.

Just as Marin--Wagner did in~\cite[Thm.1.1]{MW}, we can prove that this set
of 20 elements is linearly independent over $\BQ(t_0,t_1)$ by using the map
$\rho_{\LG}: C_3 \to \End(V^{\otimes 3})$ on a 4-dimensional vector space $V$
and check that the system of $4^6=4096$ linear equations in $20$ unknowns
with coefficients in the field $\BQ(t_0,t_1)$ has a unique solution, namely zero.
\end{proof}

\subsection{The \texorpdfstring{$n=4$}{n = 4} case}

We prove Theorem~\ref{thm.2} for words $\beta \in B_4$ inductively on
$\#_{3}(\beta)$, where $\#_{3}(\beta)$ is the sum of the number of $s_{3}$ and of
the number of $\hs_{3}$ that appear in the expression of $\beta$ in terms Artin
generators. Moreover, like in the three strand case, because of \eqref{R1}, we
need only to consider 
\be
\beta = s_{i_1}^{\varepsilon_1} s_{i_2}^{\varepsilon_2} s_{i_3}^{\varepsilon_3}
\ldots  \text{ , } \varepsilon_i = \pm 1.
\ee

The base case of the induction -- where the result is obviously true -- is when
$\#_{3}(\beta) \leq 1$.

Now consider a word $\b$ such that $\#_{3}(\beta) \geq 2$. We can write
$\beta = x s_{3}^{\pm 1} w s_{3}^{\pm 1} y$ with $w,y \in \la s_1, s_2\ra \cong B_3$
and $x\in B_4$. Since $w\in B_3$, we may express it in the generating set of
Lemma~\ref{threestrand}:
\begin{itemize}
\item
  either $w$ has at most one $s_2^{\pm 1}$,
\item
  or $w = \hs_2 s_1 \hs_2$.
\end{itemize}
\ovalbox{Case 1: $\#_{2}(w) = 0$.} In this case, $w$ is a power of $s_1$ and
$\b=xs_{3}^{\pm 1}ws_{3}^{\pm 1}y=xws_{3}^{\pm 1}s_{3}^{\pm 1}y$. The number of
$s_{3}^{\pm 1}$ used to write $\beta$ decreases thanks to \eqref{R2} or a
straightforward simplification of inverses. So we have the result in this case
inductively.
\\
\ovalbox{Case 2: $\#_{2}(w) = 1$.} We can write $w=s_1^\ve s_2^{\pm1} s_1^\g$ for
some $\varepsilon,\g\in\{-1,0,1\}$. Then $\beta = x' s_{3}^{\pm 1} s_{2}^{\pm 1}
s_{3}^{\pm 1} y'$ with $y' \in \la s_1, s_2\ra$ and $x'\in B_4$. We apply Lemma
\ref{threestrand} to $s_{3}^{\pm 1} s_{2}^{\pm 1} s_{3}^{\pm 1}\in \la s_2, s_3\ra
\cong B_3$, then modulo terms with fewer $s_{3}^{\pm 1}$,
\be
\label{eq.case2x'y''}
\begin{aligned}
    \beta = x' \hs_3 s_2 \hs_3 y'' && \mbox{for some $y'' \in \la s_1, s_2\ra$.}
\end{aligned}
\ee
If the left-most letter in $y''$ is $s_{1}^{\pm 1}$, then that letter commutes
with $\hs_3 s_2 \hs_3$ modulo words with fewer $s_{3}^{\pm 1}$ by applications of
\eqref{R3}. If the left-most letter is $s_{2}^{\pm 1}$, then we may apply Lemma
\ref{threestrand} once again. The sub-word $\hs_3 s_2 \hs_3 s_{2}^{\pm 1}$ reduces
in $\la s_2, s_3\ra \cong B_3$ to $\hs_3 s_2 \hs_3$ modulo words with fewer
$s_{3}^{\pm 1}$ and we have removed the left-most letter from $y'$.
Thus, an inductive argument on length of $y'$ shows
\be
\begin{aligned}
    \beta = x'' \hs_3 s_2 \hs_3 && \mbox{for some $x'' \in B_4$}
\end{aligned}
\ee
modulo words with fewer $s_{3}^{\pm 1}$.
Note that
\be
\begin{aligned}
    \#_3(x'' \hs_3 s_2 \hs_3)= 
    \#_3(x \hs_3 w \hs_3 y)&& \mbox{and}&&
    \#_3(x'')<\#_3(x \hs_3 w \hs_3 y).
\end{aligned}
\ee
Thus, using the inductive hypothesis on $x''$, it can be written as a linear
combination of words as described in Theorem~\ref{thm.2}. The different elements
in the sum can be considered independently. \\
\underline{Sub-case 2.1:} Suppose $x'' \in \la s_1, s_2\ra \cong B_3$. In this
instance, $x''$ can itself be reduced using Lemma~\ref{threestrand}. If
$x'' = s_{1}^{\varepsilon_1}s_{2}^{\varepsilon_2}s_{1}^{\varepsilon_3}$ for
$\varepsilon_i \in \{-1,0,1\}$, then modulo words with fewer $s_3^{\pm 1}$ and up
to a scalar,
\be
\begin{aligned}
\beta 
&= 
s_{1}^{\varepsilon_1}s_{2}^{\varepsilon_2}(s_{1}^{\varepsilon_3} \hs_3 s_2 \hs_3 )
\overset{\eqref{R3}}{=} 
s_{1}^{\varepsilon_1}(s_{2}^{\varepsilon_2} \hs_3 s_2 \hs_3) s_{1}^{\varepsilon_3}
\\
&\overset{\text{\ref{threestrand}}}{=} s_{1}^{\varepsilon_1}(\hs_3 s_2 \hs_3)
s_{1}^{\varepsilon_3} = s_{1}^{\varepsilon_1}(\hs_3 s_2 \hs_3 s_{1}^{\varepsilon_3})
\overset{\eqref{R3}}{=} s_{1}^{\varepsilon_1 + \varepsilon_3}\hs_3 s_2 \hs_3.
\end{aligned}
\ee
Now $\b$ is expressed in the desired form.
\\
If on the other hand $x'' = \hs_2 s_1\hs_2$, then modulo words with fewer
$s_3^{\pm 1}$ and up to a scalar,
\be
\begin{aligned}
  \beta = \hs_2 s_1 (\hs_2 \hs_3 s_2 \hs_3) \overset{\text{\ref{threestrand}}}{=}
  \hs_2 s_1 (\hs_3 s_2 \hs_3),
\end{aligned}
\ee
which now expresses $x''(\hs_3 s_2 \hs_3)$ in the form
$s_{1}^{0}s_{2}^{-1}s_{1}^{1}(\hs_3 s_2 \hs_3)$ considered previously. \\
\underline{Sub-case 2.2:} Suppose $\#_{3}(x'') = 1$. This time $x'' = u s_3^{\pm 1} v$,
with $u,v \in \la s_1, s_2\ra$.  Modulo terms with fewer $s_3^{\pm 1}$ and up to
a scalar we write:
\be
\begin{aligned}
  \beta 
  & = 
  u s_3^{\pm 1} (v \hs_3 s_2 \hs_3) \overset{\text{Sub-case 2.1}}{=} 
  u s_3^{\pm 1} s_{1}^{\varepsilon}\hs_3 s_2 \hs_3 
  = 
  u s_{1}^{\varepsilon} s_3^{\pm 1} \hs_3 s_2 \hs_3
\\ & = 
u s_{1}^{\varepsilon} (s_3^{\pm 1} \hs_3 s_2 \hs_3)
\overset{\text{\ref{threestrand}}}{=} u s_{1}^{\varepsilon} \hs_3 s_2 \hs_3
\overset{\text{Sub-case 2.1}}{=} 
 s_{1}^{\gamma} \hs_3 s_2 \hs_3.
\end{aligned}
\ee
The expression for $\b$ completes the proof in this sub-case.\\
\underline{Sub-case 2.3:} Next assume $x'' = s_1^{\delta} \hs_3 s_2 \hs_3$ with
$\delta \in\{-1, 0,1\}$. Again, modulo terms with fewer $s_3^{\pm 1}$ we have,
up to a scalar:
\be
\begin{aligned}
  \beta = s_1^{\delta} (\hs_3 s_2 \hs_3 \hs_3 s_2 \hs_3)
  \overset{\text{\ref{threestrand}}}{=} s_1^{\delta} \hs_3 s_2 \hs_3.
\end{aligned}
\ee
And we get the result in this case by structural induction. \\
\ovalbox{Case 3: $w = \hs_2 s_1 \hs_2$.} Then $\beta = x s_{3}^{\pm 1} \hs_2 s_1
\hs_2 s_{3}^{\pm 1} y$ with $y \in \la s_1, s_2\ra \cong B_3$. We need to be able
to reduce $s_{3}^{\pm 1} \hs_2 s_1 \hs_2 s_{3}^{\pm 1}$ in order to conclude in this
case. The following lemma is proven in Appendix~\ref{prooffourstrand} and does
just that.

\begin{lemma}
\label{fourstrand}
In $C_4$, the four words $s_{3}^{\pm 1} \hs_2 s_1
\hs_2 s_{3}^{\pm 1} \in B_4$ can be reduced to linear combinations of words of one
of the following types: 
\begin{itemize}
\item
  words with at most one $s_3^{\pm 1}$,
\item
  $\hs_3 s_2 \hs_3$, $s_1 \hs_3 s_2 \hs_3$ or $\hs_1 \hs_3 s_2 \hs_3$.
\end{itemize}
\end{lemma}

Using Lemma~\ref{fourstrand}, $\beta$ reduces modulo words with fewer $s_3^{\pm 1}$
and up to a scalar to:
\be
\begin{aligned}
\beta \overset{\text{\ref{fourstrand}}}{=} x s_{1}^{\varepsilon} \hs_3 s_2 \hs_3 y.
\end{aligned}
\ee
Here we have $\b$ given in the form of Case 2 \eqref{eq.case2x'y''}. Case 3 is now
proven in the same way. This proves Theorem~\ref{thm.2} in the four strand case $n=4$.

\subsection{The \texorpdfstring{$n\geq5$}{n >= 5} case}

We suppose that Theorem~\ref{thm.2} is true for some $n \geq 4$. Let us prove
that then it is also true for $B_{n+1}$. Because of \eqref{R1}, we need only to
prove the result for $\beta \in B_{n+1}$ that can be written:
\be
\begin{aligned}
  \beta = s_{i_1}^{\varepsilon_1} s_{i_2}^{\varepsilon_2} s_{i_3}^{\varepsilon_3}
  \ldots  \text{ , } \varepsilon_i = \pm 1.
\end{aligned}
\ee
Like in the four strand case, we prove the result for words by structural
induction on $\#_{n}(\beta)$.

The result is clearly true when  $\#_{n}(\beta) \leq 1$. For $\beta$ such that
$\#_{n}(\beta) \geq 2$, we can write $\beta = x s_{n}^{\pm 1} w s_{n}^{\pm 1} y$
with $w,y \in \la s_1, s_2, \ldots , s_{n-1} \ra \cong B_n$. Using the induction
hypothesis, $w$ is a $\BQ(t_0,t_1)$-linear combination of words in the form
prescribed by Theorem \ref{thm.2}. In the following we assume that $w$ is given
in the form of one of these spanning words.\\
\ovalbox{Case 1: $\#_{n}(w) \leq 1$.} Like in the four strand case, we can write
$w = a s_{n-1}^{\varepsilon} b$ with $\varepsilon \in\{-1,0, 1\}$ and
$a,b \in \la s_1, ... , s_{n-2} \ra$. Therefore
\be
\begin{aligned}
  \beta = x' s_{n}^{\pm 1} s_{n-1}^{\varepsilon} s_{n}^{\pm 1} y' &&
  \mbox{with $y' \in \la s_1, s_2, \ldots , s_{n-1} \ra$ and $x'\in B_{n+1}$.} 
\end{aligned}
\ee
Again following the ideas of the $B_4$ case, this word reduces to match the
description of it given in Theorem~\ref{thm.2}. The proof is completely similar
because $s_1, s_2, \ldots , s_{n-3}$ commute with $s_{n-1}$ and $s_{n}$.  \\
\ovalbox{Case 2: $w = u \hs_{n-1} s_{n-2} \hs_{n-1}$, $u \in \la s_1, \ldots ,
  s_{n-3} \ra$.} In this case we can write:
\be
\begin{aligned}
  \beta = x' s_{n}^{\pm 1} \hs_{n-1} s_{n-2} \hs_{n-1} s_{n}^{\pm 1} y &&
  \mbox{for some $x' \in \la s_1, s_2, \ldots , s_{n-1} \ra$.}
\end{aligned}
\ee
The conclusion follows like in the $B_4$ case.
This concludes the proof of Theorem~\ref{thm.2}.
\qed

\begin{proof}(of Corollary~\ref{cor.1})
If a link $L$ is the closure of a braid $\beta \in B_n$, each element of the
generating set $\beta'$ of the linear combination obtained through the reduction
algorithm has a simpler closure than $L$, which inductively allows the computation
of $\LG$ on any link. Let $L'=\mathrm{cl}(\b')$ denote the braid closure of $\b'$ a
vector from the generating set. Indeed:

$\bullet$
  If $\#_{n-1}(\beta') = 0$, then $L'$ is a split link. Therefore it vanishes after
  evaluation through $\LG$ (or $V_1$), see Proposition~\ref{lem.split}.

$\bullet$
 If $\#_{n-1}(\beta') = 1$, then the index of $\b'$ is reduced by a Markov move of
  type II. 

$\bullet$
  If $\b'=\beta'' \hs_{n-1} s_{n-2} \hs_{n-1}$ for some $\beta'' \in B_{n-2}$, then
  $\mathrm{cl}(\beta'' \hs_{n-1} s_{n-2} \hs_{n-1})=\mathrm{cl}(\beta''
  \hs_{n-1}^2 s_{n-2} )$. After applying  \eqref{R1} the $\LG$ (or $V_1$) invariant of $L'$
  is expressed as a linear combination of invariants associated to
  $\mathrm{cl}(\beta'' s_{n-1}^{\pm1}s_{n-2} )$, a braid with a lower index by Markov
  II, and the split link given by the closure of $\beta'' s_{n-2}\in B_{n}$.
\end{proof}

\subsection{How \texorpdfstring{\eqref{R3}}{R3} was found}
\label{sub.how}

Once \eqref{R3} is found, Lemma~\ref{lem.R123} follows by a machine
computation. But this also gives a method to find an \eqref{R3}. 
Namely, \eqref{R3} is derived from the realization of
$s_1 \hs_3 s_2 \hs_3 - \hs_3 s_2 \hs_3 s_1 \in C_4$ as a $\BQ(t_0,t_1)$-linear
combination of the 175 words in a 
spanning set for a version of $C_4$ given in~\cite{MW}.
We can use the explicit $R$-matrices of $\LG$ or $V_1$ and the
corresponding $\BQ(t_0,t_1)$-linear map $C_4 \to \mathrm{End}(V^{\otimes 4})$
for a 4-dimensional vector space $V$ to reduce this to a linear algebra
question over $\BQ(t_0,t_1)$. This results in solving a system of $4^8=65536$
sparse linear equations in $175$ variables. 

To reduce the complexity of the task we used the $R$-matrix $R_V$ for $V_1$
rather than $R_{\LG}$ since the former appeared to have simpler coefficients.
Then, we solved the sparse linear system of equations 
for a sample of $20$ different values of the pair $(t_0,t_1)$. Doing this, we found
that only $78$ of the $175$ unknown variables are rational
functions with nonzero specializations, which reflects the sparsity of the system,
and moreover, the set of these $78$ variables was the same for all attempted
specializations. We thus reduced the system of unknowns from $175$ to the $78$
ones found above, and then by the use of a computer and some by-hand eliminations,
we found the unique solution given in the appendix.

This produced a potential $80$-term \eqref{R3} skein relation that we
then checked was satisfied for both $R$-matrices involved.

A \texttt{Mathematica} program that includes the $R$-matrices of the $\LG$, $V_1$
and $\ADO_\omega$-polynomials and checks that they satisfy the skein relations is
given in~\cite{data}. 


\section{Basics of the \texorpdfstring{$\LG$}{LG},
  \texorpdfstring{$V_1$}{V1} and
  \texorpdfstring{$\ADO_\omega$}{ADO} link invariants}
\label{sec.RT}

The three polynomial invariants of links that we study in our paper, namely
the $\LG$, $V_1$ and $\ADO_\omega$ come from the well-known Reshetikhin--Turaev
construction~\cite{RT,Tu:book} applied to enhanced matrices given below.

Instead of repeating definitions and notations from previous works and arguments
that we will not use, we comment briefly how these invariants are defined
following~\cite{GHKST} and references therein. 

A rigid $R$-matrix leads to invariants of long knots~\cite{Kas2}
and Nichols algebras with automorphisms (or suitable finite dimensional quotients
thereof) produce rigid $R$-matrices and hence invariants of long knots~\cite{GK:multi}.

All three polynomial invariants come from rigid $R$-matrices, and in fact from
enhanced ones, in the sense of Ohtsuki and Turaev~\cite{Oht, Tu:YB}.
For a precise definition see~\cite[Sec.2]{GHKST}, 
where an extension to tangles is given, and
a comparison of the various definitions is also discussed.

A common feature of the invariants of tangles given by the three polynomials
that we study is that they vanish when evaluated to closed links, since in a sense
all three are fermionic invariants.

To overcome this problem and define a nontrivial invariant of oriented links,
one cuts one component to obtain a $(1,1)$-tangle, and then shows that the
invariant of $(1,1)$-tangles is a scalar (so-called property $(P_1)$
in ~\cite[Sec.1.1]{GHKST}), and then that an invariant of a $(2,2)$-tangle
is unchanged if we close it on one or the other side (property $(P_2)$). 

All three tangle invariants satisfy properties $(P_1)$ and $(P_2)$ and consequently
give well-defined invariants of oriented links. Although these link invariants are
highly nontrivial, they do vanish on split links. 

Note that in dealing with the Links--Gould invariant of links, we stick to
conventions used by Ishii for example in \cite{Ish}. In doing so,
$\LG_L(p^{-2}, p^{2} q^{2})$ with $p = q^{\alpha}$ coincides with the Links-Gould
invariant from \cite{DWKL}.

Having discussed the basic properties of the three invariants of interest,
we give their $R$-matrices in Appendix~\ref{sec.Rmat}, and their
enhancements here:

\be
\label{3h}
\begin{aligned}
h_{\LG} & =\diag(t_0^{-1}, -t_1, -t_0^{-1}, t_1) \in \End(W) \\
h_{V}   & =\diag(-1, 1, 1, -1) \in \End(V) \\
h_{\ADO} & =\diag(t^2, \omega^2 t^2, \omega^4 t^2) \in \End(X) \,.
\end{aligned}
\ee

\begin{lemma}
$(R_{\LG},h_{\LG})$, $(R_{V},h_{V})$ and $(R_{\ADO},h_{\ADO})$ are enhanced
$R$-matrices and satisfy properties $(P_1)$ and $(P_2)$ of \cite[Sec.1.1]{GHKST}.
\end{lemma}

The statement about enhancement follows by an explicit computation. 
Regarding properties $(P_1)$ and $(P_2)$, $(R_{V},h_{V})$ satisfies them as was
shown in \cite[Sec.3]{GHKST}. So do $(R_{\LG},h_{\LG})$ and $(R_{\ADO},h_{\ADO})$
since they are defined representation theoretically via a ribbon category
and the tangle invariants are colored by a simple ambidextrous object
in the sense of \cite{GPT}.

The link invariants can be computed in terms of a braid presentation 
$\beta\in B_n$ of an oriented link $L$ as stated in~\cite[Rem.2.3]{GHKST}
and for the convenience of the reader, we reproduce here:
\be
\label{FLcut}
\begin{aligned}
\RT_{L^{\mathrm{cut}}}&=
\tr_{2,\dots, n}\left((\mathrm{id}_V\otimes
h^{\otimes (n-1)})\circ \rho_R(\b) \right)\in \End(V)
\\
\la\RT_{L^{\mathrm{cut}}}\ra
&=
\frac{1}{\dim(V)}
\tr\left((\mathrm{id}_V\otimes
h^{\otimes (n-1)})\circ \rho_R(\b)\right)\,.
\end{aligned}
\ee

As mentioned before, all three link invariants thus defined have the following
common feature.

\begin{lemma}
\label{lem.split}
The $\LG$, $V_1$ and $\ADO_\omega$ polynomials vanish on split links, and are equal to $1$ on the unknot.
\end{lemma} 

\begin{proof}
The vanishing on split links follows from the definition of the invariants
and the fact that the diagonal matrices $h_{\LG}$, $h_{V}$ and $h_{\ADO}$
have trace zero. The value of the unknot, whose long version is a single
vertical strand, is obvious. 
\end{proof}


\appendix

\section{The \texorpdfstring{$R$}{R}-matrices for the Links--Gould,
  \texorpdfstring{$\ADO_\omega$}{ADO} and
  \texorpdfstring{$V_1$}{V1} polynomials}
\label{sec.Rmat}

In this appendix we write the three $R$-matrices that we need in the paper.

Since all three $R$-matrices are sparse, we present them in the following way.
Suppose $V$ is a vector space over a field $k$ with an ordered basis
$(v_1,\dots,v_n)$ and an $R$ matrix $R \in \End(V \otimes V)$.
Abbreviating $v_{ij}=v_i \otimes v_j$ for $i,j=1,\dots n$, the $n^2 \times n^2$
matrix $R$ can be presented
as an $n \times n$ matrix $\mathsf{R}:=(R(x_{ij}))_{1 \leq i,j \leq n}$ whose
entries are $k$-linear combinations of $v_{ij}$.

We now give the three $R$-matrices, beginning 
with the Links--Gould polynomial whose $R$-matrix is defined as follows.
Consider a 4-dimensional $\BQ(t_0,t_1)$-vector space $W$ with ordered basis
$(w_1,w_2,w_3,w_4)$.
Abbreviating $w_{ij}=w_i \otimes w_j$ for $i,j=1,\dots 4$, the $R$-matrix
$\mathsf{R}_{\LG}:=((R_{\LG})(w_{ij}))_{1 \leq i,j \leq 4}$ is given by
\begin{center}
\small
\resizebox{\textwidth}{!}{
$
  \mathsf{R}_{\LG}=\left(\begin{array}{@{}cccc}
  t_0 w_{11}
  &
  t_0^{1/2} w_{21}
  &
  t_0^{1/2} w_{31}
  &
   w_{41}
  \\ 
  t_0^{1/2} w_{12} + (t_0-1)w_{21}
  &
  - w_{22}
  &
  (t_0t_1-1) w_{23}
  -t_0^{1/2}t_1^{1/2} w_{32}
  -t_0^{1/2}t_1^{1/2}Y w_{41}
  &
  t_1^{1/2} w_{42}
  \\ 
  t_0^{1/2} w_{13} + (t_0-1)w_{31}
  &
  -t_0^{1/2}t_1^{1/2} w_{23}+Y w_{41}
  &
  - w_{33}
  &
  t_1^{1/2} w_{43}
  \\ 
  w_{14}-t_0^{1/2}t_1^{1/2}Y w_{23}+Y w_{32}+Y^2 w_{41}
  &
  t_1^{1/2} w_{24}+(t_1-1) w_{42}
  &
  t_1^{1/2} w_{34}+(t_1-1)w_{43}
  &
  t_1 w_{44}
  \end{array}\right)
$}
\end{center} 
with $Y = \sqrt{(t_0-1)(1-t_1)}$. Actually, the entries in the above matrix are
in the quadratic extension $\BQ(t_0,t_1)[Y]$ of the field $\BQ(t_0,t_1)$ but this
plays no important role in our arguments. 

Next we give the $R$-matrix of the $V_1$-polynomial whose explicit computation was
discussed in~\cite{GK:multi} and further studied in \cite{GHKST}. We consider a
4-dimensional $\BQ(t_0,t_1)$-vector space $V$ with ordered
basis $(v_1,v_2,v_3,v_4)$.
As before, with $v_{ij}=v_i \otimes v_j$, the $R$-matrix 
$\mathsf{R}_{V,r}:=(R_{V,r}(v_{ij}))_{1 \leq i,j \leq 4}$ is given by
\begin{center}
\small
\resizebox{\textwidth}{!}{
$
\mathsf{R}_{V,r}=
\left(
\begin{array}{cccc}
- v_{11} 
& 
-t_0 v_{21}
& 
-t_1 v_{31}
&
-t_0t_1 v_{41}
\\
- v_{12}
+
(t_0-1) v_{21}
& 
t_0 v_{22}
& 
-rt_1 v_{32}
+
(t_0-1)t_1 v_{41}
&
rt_0t_1 v_{42}
\\
- v_{13}
+
(t_1-1) v_{31}
&
-r^{-1}t_1^{-1} v_{23}
+
r^{-1}(1-t_0) v_{41}
&
t_1 v_{33}
&
r^{-1} v_{43}
\\
\begin{bmatrix}
- v_{14}
+
(t_1^{-1}-1) v_{23}
\\
+
r(t_1-1) v_{32}
+
(t_0+t_1-2) v_{41}
\end{bmatrix}
&
r^{-1}t_1^{-1}v_{24}
+
(t_0-1) v_{42}
&
rt_1 v_{34}
+
(t_1-1) v_{43}
&
- v_{44}
\end{array}
\right).$}
\end{center} 
\vspace{0.25cm}

Taking $r=1$, the $R$-matrix $R_V:=R_{V,1}$ has an enhancement as was explained
in~\cite[Sec.3]{GHKST}. 

Lastly, we give the $R$-matrix used to define the $\ADO_\omega$ invariant of links.
With $\omega = e^{2\pi i/6}$, consider a 3-dimensional $\BQ(\omega,t)$-vector
space $X$ with an ordered basis $(x_0, x_1, x_2)$.
%
With $x_{ij}=x_i \otimes x_j$, the $3 \times 3$ $R$-matrix
$\mathsf{R}_{\ADO}:=(R_{\ADO}(x_{ij}))_{1 \leq i,j \leq 3}$ is given by

\begin{equation*}
\mathsf{R}_{\ADO}=
\left(
\begin{array}{ccc}
t^2 x_{00} 
& 
(t^2-1) x_{01}
+ 
t x_{10}
& 
(t^2-1)(1-\omega^2t^{-2}) x_{02}
+
(t^{-1}+\omega t) x_{11}
+
x_{20}
\\
t x_{01}
& 
(t-t^{-1}) x_{02}
+
\omega^2 x_{11}
& 
(\omega^2 t^{-2}-1) x_{12}
+
-\omega t^{-1} x_{21}
\\
x_{02}
&
-\omega t^{-1} x_{12}
&
\omega^2 t^{-2} x_{22}
\end{array}
\right).
\end{equation*}


\section{The coefficients of the \texorpdfstring{\eqref{R3}}{R3} skein relation}
\label{sec.R123}

In this section we give the coefficients of the \eqref{R3}-skein relation. 

\begin{tiny}
\begin{align*}
a_1 & = -\frac{(t_1 - 1) (t_0 - 1) (-t_1 - t_0 - 2 t_1 t_0
- t_1^2 t_0 - t_1 t_0^2 + t_1^3 t_0^2 + t_1^2 t_0^3)}{
t_1 t_0 (t_1 + t_0) (t_1 t_0 - 1) (1 + t_1 t_0)} \,,
\\
a_2 & = -\frac{(t_1 - 1) (t_0 - 1) (t_1 + t_0 + 2 t_1 t_0)}{
(t_1 + t_0) (t_1 t_0 - 1) (1 + t_1 t_0)} \,,
\\
a_3 & = \frac{(t_1 - 1) (t_0 - 1) (-t_1 - t_0 - 2 t_1 t_0 - t_1^2 t_0
- t_1 t_0^2 + t_1^3 t_0^2 + t_1^2 t_0^3)}{t_1 t_0 (t_1 + t_0) (t_1 t_0 - 1)
(1 + t_1 t_0)} \,,
\\
a_4 & = \frac{(t_1 - 1) (t_0 - 1) (t_1 + t_0 + 2 t_1 t_0)}{(t_1 + t_0)
(t_1 t_0 - 1) (1 + t_1 t_0)} \,,
\\
a_5 & = \frac{1 + t_1 t_0 + t_1^2 t_0 + t_1 t_0^2}{t_1 (1 + t_1) t_0 (1 + t_0)
(1 + t_1 t_0)} \,,
\\
a_6 & = -\frac{t_1 + t_0 + t_1 t_0 + t_1^2 t_0^2}{t_1 (1 + t_1) t_0 (1 + t_0)
(1 + t_1 t_0)} \,,
\\
a_7 & = -\frac{(1 + t_1 t_0 + t_1^2 t_0 + t_1 t_0^2) (t_1 - t_1^2 + t_0 - t_1 t_0
+ t_1^2 t_0 - t_0^2 + t_1 t_0^2 + t_1^2 t_0^2)}{t_1 (1 + t_1) t_0 (1 + t_0)
(t_1 + t_0) (1 + t_1 t_0)} \,,
\\
a_8 & = \frac{(t_1 + t_0 + t_1 t_0 + t_1^2 t_0^2) (t_1 - t_1^2 + t_0 - t_1 t_0
+ t_1^2 t_0 - t_0^2 + t_1 t_0^2 + t_1^2 t_0^2)}{t_1 (1 + t_1) t_0 (1 + t_0)
(t_1 + t_0) (1 + t_1 t_0)} \,,
\\
a_9 & = -\frac{1 + t_1 t_0 + t_1^2 t_0 + t_1 t_0^2}{t_1 (1 + t_1) t_0 (1 + t_0)
(1 + t_1 t_0)} \,,
\\
a_{10} & = \frac{(1 + t_1 t_0 + t_1^2 t_0 + t_1 t_0^2) (t_1 - t_1^2 + t_0
- t_1 t_0 + t_1^2 t_0 - t_0^2 + t_1 t_0^2 + t_1^2 t_0^2)}{t_1 (1 + t_1) t_0 (1 + t_0)
(t_1 + t_0) (1 + t_1 t_0)} \,,
\\
a_{11} & = \frac{t_1 + t_0 + t_1 t_0 + t_1^2 t_0^2}{t_1 (1 + t_1) t_0 (1 + t_0)
(1 + t_1 t_0)} \,,
\\
a_{12} & = -\frac{(t_1 + t_0 + t_1 t_0 + t_1^2 t_0^2) (t_1 - t_1^2 + t_0 - t_1 t_0
+ t_1^2 t_0 - t_0^2 + t_1 t_0^2 + t_1^2 t_0^2)}{t_1 (1 + t_1) t_0 (1 + t_0)
(t_1 + t_0) (1 + t_1 t_0)} \,,
\\
a_{13} & = \frac{-t_1^2 - t_1 t_0 - t_1^2 t_0 - t_0^2 - t_1 t_0^2 + t_1^2 t_0^2}{t_1
(1 + t_1) t_0 (1 + t_0) (t_1 + t_0) (1 + t_1 t_0)} \,,
\\
a_{14} & = \frac{-t_1 + t_1^2 - t_0 + t_1 t_0 + t_1^2 t_0 + t_0^2 + t_1 t_0^2
- t_1^2 t_0^2 + 2 t_1^3 t_0^2 + t_1^4 t_0^2 + 2 t_1^2 t_0^3 + 2 t_1^3 t_0^3
+ t_1^4 t_0^3 + t_1^2 t_0^4 + t_1^3 t_0^4}{t_1 (1 + t_1) t_0 (1 + t_0) (t_1 + t_0)
(1 + t_1 t_0)} \,,
\\
a_{15} & = -\frac{-t_1^3 - t_1 t_0 - t_1^4 t_0 - t_1^2 t_0^2 - t_0^3 + t_1^4 t_0^3
- t_1 t_0^4 + t_1^3 t_0^4}{t_1 (1 + t_1) t_0 (1 + t_0) (t_1 + t_0) (1 + t_1 t_0)} \,,
\\
a_{16} & = -\frac{-t_1^2 + t_1^3 - t_1 t_0 + t_1^2 t_0 + t_1^3 t_0 - t_0^2
+ t_1 t_0^2 + 3 t_1^2 t_0^2 + t_1^4 t_0^2 + t_0^3 + t_1 t_0^3 + 2 t_1^3 t_0^3
+ t_1^4 t_0^3 + t_1^2 t_0^4 + t_1^3 t_0^4}{t_1 (1 + t_1) t_0 (1 + t_0)
(t_1 + t_0) (1 + t_1 t_0)} \,,
\\
a_{17} & = -\frac{1 + t_1 t_0 + t_1^2 t_0 + t_1 t_0^2}{(1 + t_1) (1 + t_0)
(t_1 + t_0) (1 + t_1 t_0)} \,,
\\
a_{18} & = -\frac{(t_1 + t_0 - 1) (1 + t_1 t_0 + t_1^2 t_0 + t_1 t_0^2)}{(1 + t_1)
(1 + t_0) (t_1 + t_0) (1 + t_1 t_0)} \,,
\\
a_{19} & = \frac{t_1 + t_0 + t_1 t_0 + t_1^2 t_0^2}{(1 + t_1) (1 + t_0) (t_1 + t_0)
(1 + t_1 t_0)} \,,
\\
a_{20} & = \frac{(t_1 + t_0 - 1) (t_1 + t_0 + t_1 t_0 + t_1^2 t_0^2)}{(1 + t_1)
(1 + t_0) (t_1 + t_0) (1 + t_1 t_0)} \,,
\\
a_{21} & = -\frac{t_1 + t_0 + t_1 t_0 + t_1^2 t_0^2}{t_1 (1 + t_1) t_0 (1 + t_0)
(t_1 + t_0) (1 + t_1 t_0)} \,,
\\
a_{22} & = -\frac{-t_1^2 - t_1 t_0 - t_1^4 t_0 - t_0^2 + t_1^2 t_0^2 - t_1^3 t_0^2
- t_1^2 t_0^3 + t_1^4 t_0^3 - t_1 t_0^4 + t_1^3 t_0^4}{t_1 (1 + t_1) t_0 (1 + t_0)
(t_1 + t_0) (1 + t_1 t_0)} \,,
\\
a_{23} & = \frac{t_1 - t_1^2 + t_0 - t_1 t_0 - t_1^4 t_0 - t_0^2 + t_1^2 t_0^2
- 2 t_1^3 t_0^2 - 2 t_1^2 t_0^3 + t_1^4 t_0^3 - t_1 t_0^4 + t_1^3 t_0^4}{t_1
(1 + t_1) t_0 (1 + t_0) (t_1 + t_0) (1 + t_1 t_0)} \,,
\\
a_{24} & = -\frac{(t_1 + t_0 - 1) (-t_1^2 - t_1 t_0 - t_1^2 t_0 - t_0^2 - t_1 t_0^2
+ t_1^2 t_0^2)}{t_1 (1 + t_1) t_0 (1 + t_0) (t_1 + t_0) (1 + t_1 t_0)} \,,
\\
a_{25} & = \frac{1 + t_1 t_0 + t_1^2 t_0 + t_1 t_0^2}{(1 + t_1) (1 + t_0) (t_1 + t_0)
(1 + t_1 t_0)} \,,
\\
a_{26} & = -\frac{t_1 + t_0 + t_1 t_0 + t_1^2 t_0^2}{(1 + t_1) (1 + t_0) (t_1 + t_0)
(1 + t_1 t_0)} \,,
\\
a_{27} & = \frac{(t_1 + t_0 - 1) (1 + t_1 t_0 + t_1^2 t_0 + t_1 t_0^2)}{(1 + t_1)
(1 + t_0) (t_1 + t_0) (1 + t_1 t_0)} \,,
\\
a_{28} & = -\frac{(t_1 + t_0 - 1) (t_1 + t_0 + t_1 t_0 + t_1^2 t_0^2)}{(1 + t_1)
(1 + t_0) (t_1 + t_0) (1 + t_1 t_0)} \,,
\\
a_{29} & = \frac{(1 + t_1) (1 + t_0)}{(t_1 + t_0) (t_1 t_0 - 1) (1 + t_1 t_0)} \,,
\\
a_{30} & = \frac{-1 - 2 t_1 + t_1^2 - 2 t_0 + t_1 t_0 - t_1^3 t_0 + t_0^2
- 2 t_1^2 t_0^2 + t_1^3 t_0^2 - t_1 t_0^3 + t_1^2 t_0^3}{(t_1 + t_0) (t_1 t_0 - 1)
(1 + t_1 t_0)} \,,
\\
a_{31} & = \frac{(1 + t_1 t_0 + t_1^2 t_0 + t_1 t_0^2) (-t_1^2 - t_1 t_0 - t_1^2 t_0
- t_0^2 - t_1 t_0^2 + t_1^2 t_0^2)}{t_1 (1 + t_1) t_0 (1 + t_0) (t_1 + t_0)
(t_1 t_0 - 1) (1 + t_1 t_0)} \,,
\\
a_{32} & = -\frac{(t_1 + t_0 + t_1 t_0 + t_1^2 t_0^2) (-t_1^2 - t_1 t_0 - t_1^2 t_0
- t_0^2 - t_1 t_0^2 + t_1^2 t_0^2)}{t_1 (1 + t_1) t_0 (1 + t_0) (t_1 + t_0)
(t_1 t_0 - 1) (1 + t_1 t_0)} \,,
\\
a_{33} & = -\frac{t_1 + t_0 + t_1 t_0 + t_1^2 t_0 + t_1^3 t_0 + t_1 t_0^2
+ 4 t_1^2 t_0^2 - t_1^4 t_0^2 + t_1 t_0^3 - t_1^3 t_0^3 + t_1^4 t_0^3 + t_1^5 t_0^3
- t_1^2 t_0^4 + t_1^3 t_0^4 + 2 t_1^4 t_0^4 + t_1^5 t_0^4 + t_1^3 t_0^5
+ t_1^4 t_0^5}{t_1 (1 + t_1) t_0 (1 + t_0) (t_1 + t_0) (t_1 t_0 - 1) (1 + t_1 t_0)}
\,,
\\
a_{34} & = \frac{t_1^2 + 3 t_1 t_0 + 2 t_1^2 t_0 - t_1^3 t_0 + t_0^2 + 2 t_1 t_0^2
+ t_1^3 t_0^2 - t_1 t_0^3 + t_1^2 t_0^3 + t_1^3 t_0^3 + t_1^5 t_0^3 + 2 t_1^4 t_0^4
+ t_1^5 t_0^4 + t_1^3 t_0^5 + t_1^4 t_0^5}{t_1 (1 + t_1) t_0 (1 + t_0) (t_1 + t_0)
(t_1 t_0 - 1) (1 + t_1 t_0)} \,,
\\
a_{35} & = \frac{(1 + t_1 t_0 + t_1^2 t_0 + t_1 t_0^2) (t_1^2 + t_1 t_0 + t_0^2
- t_1^2 t_0^2 + t_1^3 t_0^2 + t_1^2 t_0^3)}{t_1 (1 + t_1) t_0 (1 + t_0) (t_1 + t_0)
(t_1 t_0 - 1) (1 + t_1 t_0)} \,,
\\
a_{36} & = -\frac{(t_1 + t_0 + t_1 t_0 + t_1^2 t_0^2) (t_1^2 + t_1 t_0 + t_0^2
- t_1^2 t_0^2 + t_1^3 t_0^2 + t_1^2 t_0^3)}{t_1 (1 + t_1) t_0 (1 + t_0) (t_1 + t_0)
(t_1 t_0 - 1) (1 + t_1 t_0)} \,,
\\
a_{37} & = -\frac{t_1 + t_0 + t_1 t_0 + t_1^2 t_0^2}{t_1 t_0 (t_1 + t_0) (t_1 t_0 - 1)
(1 + t_1 t_0)} \,,
\\
a_{38} & = -\frac{(1 + t_1 t_0 + t_1^2 t_0 + t_1 t_0^2) (-t_1^2 - t_1 t_0 - t_1^2 t_0
- t_0^2 - t_1 t_0^2 + t_1^2 t_0^2)}{t_1 (1 + t_1) t_0 (1 + t_0) (t_1 + t_0)
(t_1 t_0 - 1) (1 + t_1 t_0)} \,,
\\
a_{39} & = \frac{t_1 + t_0 + t_1 t_0 + 2 t_1^2 t_0 + t_1^3 t_0 - t_1^4 t_0
+ 2 t_1 t_0^2 + 4 t_1^2 t_0^2 - t_1^3 t_0^2 + t_1^5 t_0^2 + t_1 t_0^3 - t_1^2 t_0^3
+ t_1^3 t_0^3 + 3 t_1^4 t_0^3 - t_1 t_0^4 + 3 t_1^3 t_0^4 - t_1^5 t_0^4 + t_1^2 t_0^5
- t_1^4 t_0^5}{t_1 (1 + t_1) t_0 (1 + t_0) (t_1 + t_0) (t_1 t_0 - 1) (1 + t_1 t_0)}
\,,
\\
a_{40} & = -\frac{(1 + t_1 t_0 + t_1^2 t_0 + t_1 t_0^2) (t_1^2 + t_1 t_0 + t_0^2
- t_1^2 t_0^2 + t_1^3 t_0^2 + t_1^2 t_0^3)}{t_1 (1 + t_1) t_0 (1 + t_0) (t_1 + t_0)
(t_1 t_0 - 1) (1 + t_1 t_0)} \,,
\\
a_{41} & = \frac{t_1^2 + 3 t_1 t_0 - t_1^2 t_0 + t_0^2 - t_1 t_0^2 + t_1^2 t_0^2
+ t_1^3 t_0^2 - t_1^4 t_0^2 + t_1^2 t_0^3 - 2 t_1^3 t_0^3 + t_1^4 t_0^3 - t_1^2 t_0^4
+ t_1^3 t_0^4}{t_1 t_0 (t_1 + t_0) (t_1 t_0 - 1) (1 + t_1 t_0)} \,,
\\
a_{42} & = \frac{(t_1 + t_0 + t_1 t_0 + t_1^2 t_0^2) (-t_1^2 - t_1 t_0 - t_1^2 t_0
- t_0^2 - t_1 t_0^2 + t_1^2 t_0^2)}{t_1 (1 + t_1) t_0 (1 + t_0) (t_1 + t_0)
(t_1 t_0 - 1) (1 + t_1 t_0)} \,,
\\
a_{43} & = -\frac{t_1^2 + 3 t_1 t_0 + 2 t_1^2 t_0 - t_1^3 t_0 + t_0^2 + 2 t_1 t_0^2
+ 2 t_1^3 t_0^2 + t_1^4 t_0^2 - t_1 t_0^3 + 2 t_1^2 t_0^3 + 3 t_1^3 t_0^3
+ t_1^2 t_0^4}{t_1 (1 + t_1) t_0 (1 + t_0) (t_1 + t_0) (t_1 t_0 - 1) (1 + t_1 t_0)}
\,,
\\
a_{44} & = \frac{(t_1 + t_0 + t_1 t_0 + t_1^2 t_0^2) (t_1^2 + t_1 t_0 + t_0^2
- t_1^2 t_0^2 + t_1^3 t_0^2 + t_1^2 t_0^3)}{t_1 (1 + t_1) t_0 (1 + t_0) (t_1 + t_0)
(t_1 t_0 - 1) (1 + t_1 t_0)} \,,
\\
a_{45} & = \frac{(1 + t_1 t_0 + t_1^2 t_0 + t_1 t_0^2) (t_1 + t_0 + t_1 t_0
+ t_1^2 t_0^2)}{(1 + t_1) (1 + t_0) (t_1 + t_0) (t_1 t_0 - 1) (1 + t_1 t_0)} \,,
\\
a_{46} & = -\frac{(t_1 + t_0 + t_1 t_0 + t_1^2 t_0^2)^2}{(1 + t_1) (1 + t_0)
(t_1 + t_0) (t_1 t_0 - 1) (1 + t_1 t_0)} \,,
\\
a_{47} & = \frac{1 + t_1 t_0 + t_1^2 t_0 + t_1 t_0^2}{(1 + t_1) (1 + t_0) (t_1 + t_0)
(1 + t_1 t_0)} \,,
\\
a_{48} & = -\frac{t_1 + t_0 + t_1 t_0 + t_1^2 t_0^2}{(1 + t_1) (1 + t_0) (t_1 + t_0)
(1 + t_1 t_0)} \,,
\\
a_{49} & = -\frac{(t_1 + t_0 + 2 t_1 t_0) (1 + t_1 t_0 + t_1^2 t_0 + t_1 t_0^2)}{
(1 + t_1) (1 + t_0) (t_1 + t_0) (t_1 t_0 - 1) (1 + t_1 t_0)} \,,
\\
a_{50} & = \frac{(t_1 + t_0 + 2 t_1 t_0) (t_1 + t_0 + t_1 t_0 + t_1^2 t_0^2)}{
(1 + t_1) (1 + t_0) (t_1 + t_0) (t_1 t_0 - 1) (1 + t_1 t_0)} \,,
\\
a_{51} & = \frac{(1 + t_1 t_0 + t_1^2 t_0 + t_1 t_0^2) (-t_1^2 - t_1 t_0 - t_1^2 t_0
- t_0^2 - t_1 t_0^2 + t_1^2 t_0^2)}{t_1 (1 + t_1) t_0 (1 + t_0) (t_1 + t_0)
(t_1 t_0 - 1) (1 + t_1 t_0)} \,,
\\
a_{52} & = -\frac{(t_1 + t_0 + t_1 t_0 + t_1^2 t_0^2) (-t_1^2 - t_1 t_0 - t_1^2 t_0
- t_0^2 - t_1 t_0^2 + t_1^2 t_0^2)}{t_1 (1 + t_1) t_0 (1 + t_0) (t_1 + t_0)
(t_1 t_0 - 1) (1 + t_1 t_0)} \,,
\\
a_{53} & = -\frac{(t_1 + t_0 + t_1 t_0) (1 + t_1 t_0 + t_1^2 t_0 + t_1 t_0^2)}{
t_1 (1 + t_1) t_0 (1 + t_0) (t_1 + t_0) (1 + t_1 t_0)} \,,
\\
a_{54} & = \frac{(t_1 + t_0 + t_1 t_0) (t_1 + t_0 + t_1 t_0 + t_1^2 t_0^2)}{
t_1 (1 + t_1) t_0 (1 + t_0) (t_1 + t_0) (1 + t_1 t_0)} \,,
\\
a_{55} & = \frac{(1 + t_1 t_0 + t_1^2 t_0 + t_1 t_0^2) (t_1^2 + t_1^2 t_0 + t_0^2
+ t_1 t_0^2)}{t_1 (1 + t_1) t_0 (1 + t_0) (t_1 + t_0) (t_1 t_0 - 1) (1 + t_1 t_0)} \,,
\\
a_{56} & = -\frac{(t_1^2 + t_1^2 t_0 + t_0^2 + t_1 t_0^2) (t_1 + t_0 + t_1 t_0
+ t_1^2 t_0^2)}{t_1 (1 + t_1) t_0 (1 + t_0) (t_1 + t_0) (t_1 t_0 - 1)
(1 + t_1 t_0)} \,,
\\
a_{57} & = -\frac{(1 + t_1 t_0 + t_1^2 t_0 + t_1 t_0^2) (t_1 + t_0 + t_1 t_0
+ t_1^2 t_0^2)}{(1 + t_1) (1 + t_0) (t_1 + t_0) (t_1 t_0 - 1) (1 + t_1 t_0)} \,, 
\\
a_{58} & = \frac{(t_1 + t_0 + t_1 t_0 + t_1^2 t_0^2)^2}{(1 + t_1) (1 + t_0)
(t_1 + t_0) (t_1 t_0 - 1) (1 + t_1 t_0)} \,,
\\
a_{59} & = -\frac{1 + t_1 t_0 + t_1^2 t_0 + t_1 t_0^2}{(1 + t_1) (1 + t_0)
(t_1 + t_0) (1 + t_1 t_0)} \,,
\\
a_{60} & = \frac{t_1 + t_0 + t_1 t_0 + t_1^2 t_0^2}{(1 + t_1) (1 + t_0)
(t_1 + t_0) (1 + t_1 t_0)} \,,
\\
a_{61} & = \frac{(t_1 + t_0 + 2 t_1 t_0) (1 + t_1 t_0 + t_1^2 t_0 + t_1 t_0^2)}{
(1 + t_1) (1 + t_0) (t_1 + t_0) (t_1 t_0 - 1) (1 + t_1 t_0)} \,,
\\
a_{62} & = -\frac{(t_1 + t_0 + 2 t_1 t_0) (t_1 + t_0 + t_1 t_0 + t_1^2 t_0^2)}{
(1 + t_1) (1 + t_0) (t_1 + t_0) (t_1 t_0 - 1) (1 + t_1 t_0)} \,,
\\
a_{63} & = -\frac{(1 + t_1 t_0 + t_1^2 t_0 + t_1 t_0^2) (-t_1^2 - t_1 t_0 - t_1^2 t_0
- t_0^2 - t_1 t_0^2 + t_1^2 t_0^2)}{t_1 (1 + t_1) t_0 (1 + t_0) (t_1 + t_0)
(t_1 t_0 - 1) (1 + t_1 t_0)} \,,
\\
a_{64} & = \frac{(t_1 + t_0 + t_1 t_0 + t_1^2 t_0^2) (-t_1^2 - t_1 t_0 - t_1^2 t_0
- t_0^2 - t_1 t_0^2 + t_1^2 t_0^2)}{t_1 (1 + t_1) t_0 (1 + t_0) (t_1 + t_0)
(t_1 t_0 - 1) (1 + t_1 t_0)} \,,
\\
a_{65} & = \frac{(t_1 + t_0 + t_1 t_0) (1 + t_1 t_0 + t_1^2 t_0 + t_1 t_0^2)}{
t_1 (1 + t_1) t_0 (1 + t_0) (t_1 + t_0) (1 + t_1 t_0)} \,,
\\
a_{66} & = -\frac{(t_1 + t_0 + t_1 t_0) (t_1 + t_0 + t_1 t_0 + t_1^2 t_0^2)}{
t_1 (1 + t_1) t_0 (1 + t_0) (t_1 + t_0) (1 + t_1 t_0)} \,,
\\
a_{67} & = -\frac{(1 + t_1 t_0 + t_1^2 t_0 + t_1 t_0^2) (t_1^2 + t_1^2 t_0 + t_0^2
+ t_1 t_0^2)}{t_1 (1 + t_1) t_0 (1 + t_0) (t_1 + t_0) (t_1 t_0 - 1) (1 + t_1 t_0)} \,,
\\
a_{68} & = \frac{(t_1^2 + t_1^2 t_0 + t_0^2 + t_1 t_0^2) (t_1 + t_0 + t_1 t_0
+ t_1^2 t_0^2)}{t_1 (1 + t_1) t_0 (1 + t_0) (t_1 + t_0) (t_1 t_0 - 1)
(1 + t_1 t_0)} \,,
\\
a_{69} & = \frac{1}{t_1 + t_0} \,,
\\
a_{70} & = \frac{t_1 t_0}{1 + t_1 t_0} \,,
\\
a_{71} & = \frac{(t_1 - 1) (1 + t_1) (t_0 - 1) (1 + t_0)}{(t_1 + t_0) (t_1 t_0 - 1)
(1 + t_1 t_0)} \,,
\\
a_{72} & = -\frac{1}{1 + t_1 t_0} \,,
\\
a_{73} & = \frac{t_1^2 + t_0^2 - 2}{(t_1 + t_0) (t_1 t_0 - 1) (1 + t_1 t_0)} \,,
\\
a_{74} & = -\frac{1}{t_1 + t_0} \,,
\\
a_{75} & = -\frac{t_1 t_0}{1 + t_1 t_0} \,,
\\
a_{76} & = -\frac{(t_1 - 1) (1 + t_1) (t_0 - 1) (1 + t_0)}{(t_1 + t_0)
(t_1 t_0 - 1) (1 + t_1 t_0)} \,,
\\
a_{77} & = \frac{1}{1 + t_1 t_0} \,,
\\
a_{78} & = -\frac{t_1^2 + t_0^2 - 2}{(t_1 + t_0) (t_1 t_0 - 1) (1 + t_1 t_0)} \,.
\end{align*}
\end{tiny}


\section{Proof of Lemma~\ref{threestrand}}
\label{proofthreestrand}

Here we prove Lemma~\ref{threestrand}. Note that it is enough to prove the result
when the word has at most 4 letters. The general case follows by induction on the
length of the word. Now, recall that for $\beta \in B_3$, because of \eqref{R1},
we need only to consider 
$$
\beta = s_1^{\varepsilon_1} s_2^{\varepsilon_2} s_1^{\varepsilon_3} \ldots \text{ or }
\beta = s_2^{\varepsilon_1} s_1^{\varepsilon_2} s_2^{\varepsilon_3} \ldots \text{ , }
\varepsilon_i = \pm 1.
$$

So, taking all reductions into account, it is enough to prove the statement for
the following list of $16+16+8 = 40$ words:
$$
s_1^{\pm 1} s_2^{\pm 1} s_1^{\pm 1}s_2^{\pm 1} \,;\,
s_2^{\pm 1} s_1^{\pm 1}s_2^{\pm 1} s_1^{\pm 1} \,;\, s_2^{\pm 1} s_1^{\pm 1}s_2^{\pm 1} .
$$
We will express these words as linear combinations of:
\begin{itemize}
\item words of length at most 3 with at most one $s_2^{\pm 1}$,
\item $\hs_2 s_1 \hs_2$.
\end{itemize}

To do that we can use the one and two letter relations from Lemmas~\ref{equivR2}
and \ref{relations}. Starting with the $32$ four letter words, let us see which
ones reduce obviously to a linear combination of words with at most one
$s_2^{\pm 1}$ and fewer letter words.\\
\begin{footnotesize}
\begin{minipage}{.45\textwidth}
\begin{align*}
&s_1 (s_2 s_1 s_2)\text{: reduces modulo braid relation and \eqref{R1},}\\
&(s_1 s_2 s_1) \hs_2\text{: reduces modulo braid relation,}\\
&(s_1 s_2 \hs_1) s_2\text{: reduces modulo braid relation and \eqref{R1},}\\
&(s_1 s_2 \hs_1) \hs_2\text{: reduces modulo braid relation,}\\
&s_1 (\hs_2 s_1 s_2)\text{: reduces modulo braid relation and \eqref{R1},}\\
&\text{\autour{$s_1 \hs_2 s_1 \hs_2$}: harder, must be studied separately,}\\
&(s_1 \hs_2 \hs_1) s_2\text{: reduces modulo braid relation and \eqref{R1},}\\
&(s_1 \hs_2 \hs_1) \hs_2\text{: reduces modulo braid relation,}\\
&(\hs_1 s_2 s_1) s_2\text{: reduces modulo braid relation,}\\
&(\hs_1 s_2 s_1) \hs_2\text{: reduces modulo braid relation and \eqref{R1},}\\
&\text{\autour{$\hs_1 s_2 \hs_1 s_2$}: harder, must be studied separately,}\\
&\hs_1 (s_2 \hs_1 \hs_2)\text{: reduces modulo braid relation and \eqref{R1},}\\
&\hs_1 (\hs_2 s_1 s_2)\text{: reduces modulo braid relation,}\\
&(\hs_1 \hs_2 s_1) \hs_2\text{: reduces modulo braid relation and \eqref{R1},}\\
&(\hs_1 \hs_2 \hs_1) s_2\text{: reduces modulo braid relation,}\\
&(\hs_1 \hs_2 \hs_1) \hs_2\text{: reduces modulo braid relation and \eqref{R1},}
\end{align*}
\end{minipage}
\quad
\begin{minipage}{.45\textwidth}
\begin{align*}
&(s_2 s_1 s_2) s_1\text{: reduces modulo braid relation and \eqref{R1},}\\
&s_2 (s_1 s_2 \hs_1)\text{: reduces modulo braid relation,}\\
&(s_2 s_1 \hs_2) s_1\text{: reduces modulo braid relation and \eqref{R1},}\\
&s_2 (s_1 \hs_2 \hs_1)\text{: reduces modulo braid relation,}\\
&s_2 (\hs_1 s_2 s_1)\text{: reduces modulo braid relation and \eqref{R1},}\\
&\text{\autour{$s_2 \hs_1 s_2 \hs_1$}: harder, must be studied separately,}\\
&(s_2 \hs_1 \hs_2) s_1\text{: reduces modulo braid relation and \eqref{R1},}\\
&(s_2 \hs_1 \hs_2) \hs_1\text{: reduces modulo braid relation,}\\
&(\hs_2 s_1 s_2) s_1\text{: reduces modulo braid relation,}\\
&(\hs_2 s_1 s_2) \hs_1\text{: reduces modulo braid relation and \eqref{R1},}\\
&\text{\autour{$\hs_2 s_1 \hs_2 s_1$}: harder, must be studied separately,}\\
&\hs_2 (s_1 \hs_2 \hs_1)\text{: reduces modulo braid relation and \eqref{R1},}\\
&\hs_2 (\hs_1 s_2 s_1)\text{: reduces modulo braid relation,}\\
&(\hs_2 \hs_1 s_2) \hs_1\text{: reduces modulo braid relation and \eqref{R1},}\\
&\hs_2 (\hs_1 \hs_2 s_1)\text{: reduces modulo braid relation,}\\
&(\hs_2 \hs_1 \hs_2) \hs_1\text{: reduces modulo braid relation and \eqref{R1}.}
\end{align*} 
\end{minipage}
\end{footnotesize}

For the 8 three letter words, let us similarly see which ones reduce directly.
\begin{align*}
s_2 s_1 s_2 &: \text{reduces modulo braid relation,}
\\
s_2 s_1 \hs_2 &: \text{reduces modulo braid relation,}
\\
s_2 \hs_1 s_2 &: \text{reduces using the equivalent version of
\eqref{R2} expressed in Lemma~\ref{equivR2},}
\\
s_2 \hs_1 \hs_2 &: \text{reduces modulo braid relation,}
\\
\hs_2 s_1 s_2 &: \text{reduces modulo braid relation,}
\\
\hs_2 s_1 \hs_2 &: \text{reduced already,}
\\
\hs_2 \hs_1 s_2 &: \text{reduces modulo braid relation,}
\\
\hs_2 \hs_1 \hs_2 &: \text{reduces modulo braid relation.}
\end{align*}

So 4 words remain to be studied more precisely:
$$
s_1 \hs_2 s_1 \hs_2 \,;\, \hs_1 s_2 \hs_1 s_2 \,;\, s_2 \hs_1 s_2 \hs_1
\,;\, \hs_2 s_1 \hs_2 s_1 .
$$
$\bullet$ \ovalbox{Case of $s_1 \hs_2 s_1 \hs_2$.} 
Let us compute $s_1 \hs_2 \cdot \eqref{R2}$ :
\begin{align*}
s_1 \hs_2 (s_1 \hs_2) & =  - s_1 \hs_2(\hs_1 s_2 s_1)
+ s_1 \hs_2(s_1 s_2 \hs_1) + s_1 \hs_2(s_1 s_2) - s_1 \hs_2(\hs_1 s_2) \\
& \phantom{sii} + s_1 \hs_2(\hs_1 \hs_2) - s_1 \hs_2(s_2 s_1)
+ s_1 \hs_2(s_2 \hs_1) + s_1 \hs_2(\hs_2 s_1) - s_1 \hs_2(\hs_2 \hs_1)\\
& \phantom{sii} + s_1 \hs_2(\hs_1 \hs_2 s_1) - s_1 \hs_2(s_1 \hs_2 \hs_1) \\
&  = - (s_1 s_1) \hs_2 (\hs_1 s_1) + s_1 (\hs_2 \hs_2) s_1 s_2
 + (s_1 s_1) s_2 \hs_1 - (s_1 s_1) \hs_2 \hs_1 \\
& \phantom{sii} + (s_1 \hs_1) \hs_2 \hs_1 - (s_1 s_1) + 1 \\
& \phantom{sii} + s_1 ((t_0^{-1} + t_1^{-1} -1) \, \hs_2 +(t_0^{-1}
+ t_1^{-1} - t_0^{-1} t_1^{-1}) \, 1 - t_0^{-1} t_1^{-1} \, s_2) \, s_1\\
&\phantom{sii} - s_1((t_0^{-1} + t_1^{-1} -1) \, \hs_2 +(t_0^{-1}
+ t_1^{-1} - t_0^{-1} t_1^{-1}) \, 1 - t_0^{-1} t_1^{-1} \, s_2)\, \hs_1\\
&\phantom{sii} + (s_1 \hs_1) \hs_2 (\hs_1 s_1) - s_1 (\hs_2 \hs_2) \hs_1 s_2 \\
&  =  - ((t_0 + t_1 -1) \, s_1 +(t_0 + t_1 - t_0 t_1) \, 1 - t_0 t_1 \, \hs_1)\,
\hs_2 \\
&\phantom{sii} +  s_1 ((t_0^{-1} + t_1^{-1} -1) \, \hs_2 +(t_0^{-1} + t_1^{-1}
- t_0^{-1} t_1^{-1}) \, 1 - t_0^{-1} t_1^{-1} \, s_2) \, s_1 s_2 \\
&\phantom{sii} + ((t_0 + t_1 -1) \, s_1 +(t_0 + t_1 - t_0 t_1) \,
1 - t_0 t_1 \, \hs_1)\,s_2 \hs_1\\
&\phantom{sii} - ((t_0 + t_1 -1) \, s_1 +(t_0 + t_1 - t_0 t_1) \,
1 - t_0 t_1 \, \hs_1)\, \hs_2 \hs_1 \\
&\phantom{sii} + \hs_2 \hs_1 
- ((t_0 + t_1 -1) \, s_1 +(t_0 + t_1 - t_0 t_1) \,
1 - t_0 t_1 \, \hs_1)\\
&\phantom{sii} + 1 + (t_0^{-1} + t_1^{-1} -1) \, s_1 \hs_2 s_1
+ (t_0^{-1} + t_1^{-1} - t_0^{-1} t_1^{-1}) \, s_1 s_1 
- t_0^{-1} t_1^{-1}\, s_1 s_2 s_1 \\
&\phantom{sii} - (t_0^{-1} + t_1^{-1} -1) \, s_1 \hs_2 \hs_1 - (t_0^{-1}
+ t_1^{-1} - t_0^{-1} t_1^{-1}) \, 1 + t_0^{-1} t_1^{-1} \, s_1 s_2 \hs_1 \\
&\phantom{sii} - \hs_2 
- s_1((t_0^{-1} + t_1^{-1} -1) \, \hs_2 +(t_0^{-1} + t_1^{-1}
- t_0^{-1} t_1^{-1}) \, 1 - t_0^{-1} t_1^{-1} \, s_2)\, \hs_1 s_2 \\
&  =  -(t_0 + t_1 -1) \, s_1 \hs_2 -(t_0 + t_1 -t_0 t_1) \, \hs_2
+ t_0 t_1 \, \hs_1 \hs_2  \\
&\phantom{sii} + (t_0^{-1} + t_1^{-1} -1) \, s_1 (\hs_2 s_1 s_2)
+ (t_0^{-1} + t_1^{-1} -t_0^{-1} t_1^{-1}) \, (s_1 s_1) s_2 - t_0^{-1} t_1^{-1}
\, s_1(s_2 s_1 s_2) \\
&\phantom{sii} + (t_0 + t_1 - 1) \, s_1 s_2 \hs_1 + (t_0 + t_1 - t_0 t_1)
\, s_2 \hs_1 - t_0 t_1 \, \hs_1 s_2 \hs_1 \\
&\phantom{sii} - (t_0 + t_1 - 1) \, s_1 \hs_2 \hs_1 - (t_0 + t_1 - t_0 t_1)
\, \hs_2 \hs_1 + t_0 t_1 \, \hs_1 \hs_2 \hs_1 
+ \hs_2 \hs_1\\
&\phantom{sii} - (t_0 + t_1 - 1) \, s_1 - (t_0 + t_1 - t_0 t_1)
\, 1 + t_0 t_1 \, \hs_1 + 1  
+ (t_0^{-1} + t_1^{-1} -1) \, s_1 \hs_2 s_1  \\
&\phantom{sii} + (t_0^{-1} + t_1^{-1} - t_0^{-1} t_1^{-1}) ((t_0 + t_1 -1)
\, s_1 +(t_0 + t_1 - t_0 t_1) \, 1 - t_0 t_1 \, \hs_1) \\
&\phantom{sii} - t_0^{-1} t_1^{-1} \, s_1 s_2 s_1 - (t_0^{-1} + t_1^{-1} - 1)
\, s_1 \hs_2 \hs_1 - (t_0^{-1} + t_1^{-1} - t_0^{-1} t_1^{-1}) \, 1 \\
&\phantom{sii} + t_0^{-1} t_1^{-1} \, s_1 s_2 \hs_1 + \hs_2
- (t_0^{-1} + t_1^{-1} - 1) \,s_1 (\hs_2 \hs_1 s_2)\\
&\phantom{sii} - (t_0^{-1} + t_1^{-1} - t_0^{-1} t_1^{-1}) \,s_2
+ t_0^{-1} t_1^{-1} \, (s_1 s_2 \hs_1) s_2 \,.
\end{align*}
If we then expand all the terms in the sum and group together those that are
multiples of the same braid word, we find that:
\begin{align*}
s_1 \hs_2 (s_1 \hs_2)
& = (t_0 + t_1 - t_0 t_1 - 1) (t_0^{-1} + t_1^{-1} - t_0^{-1} t_1^{-1} -1) \, 1 \\
& + (t_0 + t_1 -1) (t_0^{-1} + t_1^{-1} - t_0^{-1} t_1^{-1} -1) \, s_1
+ (t_0 + t_1 - t_0 t_1 -1) (t_0^{-1} + t_1^{-1} - t_0^{-1} t_1^{-1}) \, s_2 \\
& - (t_0 + t_1 - t_0 t_1 - 1) \, \hs_1 - (t_0 + t_1 - t_0 t_1 - 1) \, \hs_2 \\
& + (t_0 + t_1 -1) (t_0^{-1} + t_1^{-1} - t_0^{-1} t_1^{-1}) \, s_1 s_2
- (t_0 + t_1 -1) \, s_1 \hs_2  - (t_0 + t_1 -1) \, \hs_1 s_2 \\
& + t_0 t_1 \, \hs_1 \hs_2 - (t_0^{-1} + t_1^{-1} -1) \, s_2 s_1
+ (t_0^{-1} + t_1^{-1}) (t_0 + t_1 - t_0 t_1) \, s_2 \hs_1\\
& + (t_0^{-1} + t_1^{-1} -1) \, \hs_2 s_1 + (t_0 + t_1 - t_0 t_1^{-1}
- t_0^{-1} t_1 -1) \, \hs_2 \hs_1 - (t_0^{-1} + t_1^{-1}) \, s_1 s_2 s_1 \\
& + (t_0 + t_1) (t_0^{-1} + t_1^{-1}) \, s_1 s_2 \hs_1
+ (t_0^{-1} + t_1^{-1} -1) \, s_1 \hs_2 s_1
- (t_0^{-1} t_1 + t_0 t_1^{-1} +1) \, s_1 \hs_2 \hs_1\\
& + \hs_1 s_2 s_1 - (t_0 + t_1) \, \hs_1 s_2 \hs_1
+ (t_0 + t_1) \, \hs_1 \hs_2 \hs_1  - \hs_2 s_1 \hs_2 \,.
\end{align*}
$\bullet$ \ovalbox{Case of $\hs_1 s_2 \hs_1 s_2$.} Like in the previous case,
one can compute $\eqref{R2} \cdot s_1 s_2$ to find that:
\begin{align*}
  \hs_1 s_2 \hs_1 s_2& = (t_0 + t_1 - t_0 t_1 - 1) (t_0^{-1} + t_1^{-1}
  - t_0^{-1} t_1^{-1} -1) \, 1 + (1 - t_0^{-1} t_1^{-1})
  (t_0 + t_1 -t_0 t_1 - 1) \, s_1\\
  & + (t_0^{-1} + t_1^{-1} -2) (t_0 + t_1 - t_0 t_1 -1) \, \hs_1
  + (t_0^{-1} + t_1^{-1} - t_0^{-1} t_1^{-1} -1) (t_0 + t_1 -1) \, s_2 \\
  & - (t_0 + t_1 - t_0 t_1 -1) \, \hs_2 + (1 - t_0^{-1} t_1^{-1} )
  (t_0 + t_1 -1) \, s_1 s_2 + (1 - t_0 t_1) \, s_1 \hs_2 \\
  & + (t_0^{-1} + t_1^{-1} -2) (t_0 + t_1 -1) \, \hs_1 s_2
  - (t_0+t_1-2t_0 t_1) \, \hs_1 \hs_2 + (t_0 + t_1 - t_0 t_1) \, \hs_2 s_1 \\
  & - (t_0 + t_1 - t_0 t_1 -1) \hs_2 \hs_1 - s_1 s_2 s_1
  + (t_0 + t_1) \,  s_1 s_2 \hs_1  - (t_0 + t_1 -1) \,  s_1 \hs_2 \hs_1 \\
  & + \hs_1 s_2 s_1 - \hs_1 s_2 \hs_1 + t_0 t_1  \, \hs_1 \hs_2 \hs_1
  - t_0 t_1  \, \hs_2 s_1 \hs_2 \,.
\end{align*}
$\bullet$ \ovalbox{Case of $s_2 \hs_1 s_2 \hs_1$.} If we simplify
$s_2 \hs_1 \cdot \eqref{R2}$ we get the following expression:
\begin{align*}
  s_2 \hs_1 s_2 \hs_1& =  (t_0 + t_1 - t_0 t_1 - 1) (t_0^{-1} + t_1^{-1}
  - t_0^{-1} t_1^{-1} -1) \, 1 + (1 - t_0^{-1} t_1^{-1})
  (t_0 + t_1 -t_0 t_1 - 1) \, s_1\\
  & + (t_0^{-1} + t_1^{-1} -2) (t_0 + t_1 - t_0 t_1 -1) \, \hs_1
  + (t_0^{-1} + t_1^{-1} - t_0^{-1} t_1^{-1} -1) (t_0 + t_1 -1) \, s_2 \\
  & - (t_0 + t_1 - t_0 t_1 -1) \, \hs_2 - (t_0^{-1} + t_1^{-1}-1)
  (t_0 + t_1) \, s_1 s_2 \\
  & + (t_0^{-1} + t_1^{-1} + 1) (t_0 + t_1 -t_0 t_1) \, s_1 \hs_2
  + (t_0^{-1} + t_1^{-1} -1) (t_0 + t_1) \, \hs_1 s_2\\
  & + (1 - (t_0^{-1} + t_1^{-1} +1) (t_0 + t_1 - t_0 t_1)) \, \hs_1 \hs_2
  + (1 + t_0^{-1} t_1 + t_0 t_1^{-1} - t_0^{-1} - t_1^{-1}
+ t_0^{-1} t_1^{-1}) \, s_2 s_1 \\
& - (t_0 + t_0^{-1} + t_1 + t_1^{-1} -2) \, s_2 \hs_1
- (1 + t_0^{-1} t_1 + t_0 t_1^{-1} - t_0 - t_1 + t_0 t_1) \, \hs_2 s_1\\
& + t_0 t_1 (1 + (t_0^{-1} + t_1^{-1} -1)^2) \hs_2 \hs_1 - s_1 s_2 s_1
+ (1 - (t_0^{-1} + t_1^{-1} -1) (t_0 + t_1)) \,  s_1 s_2 \hs_1 \\
& + (t_0^{-1} + t_1^{-1} -1) (t_0 + t_1) \,  s_1 \hs_2 \hs_1
+ (t_0^{-1} + t_1^{-1}) (t_0 + t_1) \, \hs_1 s_2 s_1 - \hs_1 s_2 \hs_1\\
& - (t_0^{-1} t_1 + t_0 t_1^{-1} +1) \hs_1 \hs_2 s_1 + t_0 t_1  \,
\hs_1 \hs_2 \hs_1 - t_0 t_1  \, \hs_2 s_1 \hs_2  \,.
\end{align*}
$\bullet$ \ovalbox{Case of $\hs_2 s_1 \hs_2 s_1$.} Finally, if we write
$\hs_2 s_1 \cdot \eqref{R2}$ we get:
\begin{align*}
  \hs_2 s_1 \hs_2 s_1& =  (t_0 + t_1 - t_0 t_1 - 1) (t_0^{-1} + t_1^{-1}
  - t_0^{-1} t_1^{-1} -1) \, 1 + (t_0 + t_1 - 1) (t_0^{-1} + t_1^{-1}
  - t_0^{-1} t_1^{-1} -1) \, s_1\\
  & - (t_0 + t_1 - t_0 t_1 -1) \, \hs_1 - (t_0^{-1} + t_1^{-1}
  - t_0^{-1} t_1^{-1} -1) \, s_2\\
  & + (t_0 + t_1 - t_0 t_1 -1) (t_0^{-1} + t_1^{-1} -1) \, \hs_2
  +  s_1 s_2 - s_1 \hs_2 + \hs_1 \hs_2\\
  & - (t_0^{-1} + t_1^{-1} - t_0^{-1} t_1^{-1}) \, s_2 s_1 + s_2 \hs_1
  + (t_0 + t_1 -1) (t_0^{-1} + t_1^{-1} - 1) \, \hs_2 s_1 \\
  & - (t_0 + t_1 - t_0 t_1) \hs_2 \hs_1  + s_1 s_2 \hs_1 - s_1 \hs_2 s_1
  + \hs_1 \hs_2 s_1 -   \hs_2 s_1 \hs_2 \,.
\end{align*}

This ends the proof of Lemma~\ref{threestrand}.


\section{Proof of Lemma~\ref{fourstrand}}
\label{prooffourstrand}

Let us show that the following four words 
$$
s_{3} \hs_2 s_1 \hs_2 s_{3}\text{ , }s_{3} \hs_2 s_1 \hs_2 \hs_{3}\text{ , }
\hs_{3} \hs_2 s_1 \hs_2 s_{3}\text{ and } \hs_{3} \hs_2 s_1 \hs_2 \hs_{3} \in B_4
$$
can be reduced to linear combinations of words of one of the following types:
\begin{itemize}
\item
  \text{[Type 1]} words with at most one $s_3^{\pm 1}$, 
\item \text{[Type 2]} $\hs_3 s_2 \hs_3$, $s_1 \hs_3 s_2 \hs_3$ or
  $\hs_1 \hs_3 s_2 \hs_3$.  
\end{itemize}
To do so we write some equations that will be useful eventually.

\begin{proposition}
The following equations are true in $C_4$, modulo terms of Types 1 and 2: 
\be
\label{eq1}
s_3 \hs_2 s_1 \hs_2 s_3 = \hs_3 \hs_2 s_1 \hs_2 s_3 - s_2 s_3 \hs_2 s_1
\hs_2 s_3 + s_2 \hs_3 \hs_2 s_1 \hs_2 s_3,
\ee
\be\label{eq2}
\hs_3 \hs_2 s_1 \hs_2\hs_3 = s_2 s_3 \hs_2 s_1 \hs_2 \hs_3 - s_2 \hs_3
\hs_2 s_1 \hs_2 \hs_3 + s_3 \hs_2 s_1 \hs_2 \hs_3.
\ee
\end{proposition}

\begin{proof}
We use relation \eqref{R2} on the first two letters of certain words to find
Equations~\eqref{eq1} and \eqref{eq2}. The following equalities are written
modulo Type 1 and Type 2 terms.
\begin{align*}
(s_3 \hs_2) s_1 \hs_2 s_3 & =  -(\hs_3 s_2 s_3)s_1 \hs_2 s_3 + (s_3 s_2 \hs_3)
s_1 \hs_2 s_3 + (s_3 s_2) s_1 \hs_2 s_3\\
&\phantom{sii} -(\hs_3 s_2) s_1 \hs_2 s_3 + (\hs_3 \hs_2) s_1 \hs_2 s_3
- (s_2 s_3) s_1 \hs_2 s_3\\
&\phantom{sii} + (s_2 \hs_3) s_1 \hs_2 s_3 + (\hs_2 s_3) s_1 \hs_2 s_3
- (\hs_2 \hs_3) s_1 \hs_2 s_3 \\
&\phantom{sii} + (\hs_3 \hs_2 s_3) s_1 \hs_2 s_3 - (s_3 \hs_2 \hs_3) s_1 \hs_2 s_3 \,.
\end{align*}
Some of the terms can be simplified modulo Type 1 and Type 2 terms:
\begin{align*}
  (s_3 s_2 \hs_3) s_1 \hs_2 s_3 & = \hs_2 s_3 (s_2 s_1 \hs_2) s_3
  = \hs_2 s_3 \hs_1 s_2 s_1 s_3 = \hs_2 \hs_1 (s_3 s_2 s_3) s_1
  = \hs_2 \hs_1 s_2 s_3 s_2 s_1\,,
\\
s_3 (s_2 s_2 \hs_2) s_3 & = s_3 \hs_1 s_2 s_1 s_3 = \hs_1 (s_3 s_2 s_3) s_1
= \hs_1 s_2 s_3 s_2 s_1 \,,
\\
\hs_3 (s_2 s_1 \hs_2) s_3 & = \hs_3 \hs_1 s_2 s_1 s_3 = \hs_1 \hs_3 s_2 s_3 s_1
= \hs_1 s_2 s_3 \hs_2 s_1 \,,
\\
s_2 s_3 s_1 \hs_2 s_3 & = s_2 (s_1 s_3 \hs_2 s_3) \overset{\eqref{R3}}{=}
(s_2 s_3 \hs_2 s_3) s_1 \overset{\text{reduction in }\la s_2,s_3\ra}{=}
\hs_3 s_2 \hs_3 s_1 \overset{\eqref{R3}}{=} s_1 \hs_3 s_2 \hs_3\,,
\\
s_2 \hs_3 s_1 \hs_2 s_3 & = s_2 s_1 (\hs_3 \hs_2 s_3) = s_2 s_1 s_2 \hs_3 \hs_2 \,,
\\
\hs_2 s_3 s_1 \hs_2 s_3 & = \hs_2 (s_1 s_3 \hs_2 s_3) \overset{\eqref{R3}}{=}
(\hs_2 s_3 \hs_2 s_3)s_1 \overset{\text{reduction in }\la s_2,s_3\ra}{=}
\hs_3 s_2 \hs_3 s_1 \overset{\eqref{R3}}{=} s_1 \hs_3 s_2 \hs_3 \,,
\\
\hs_2 \hs_3 s_1 \hs_2 s_3 & = \hs_2 s_1 (\hs_3 \hs_2 s_3) = \hs_2 s_1 s_2
\hs_3 \hs_2 \,,
\\
(s_3 \hs_2 \hs_3) s_1 \hs_2 s_3 & = \hs_2 \hs_3(s_2 s_1 \hs_2)s_3 = \hs_2
\hs_3 \hs_1 s_2 s_1 s_3 = \hs_2 \hs_1 (\hs_3 s_2 s_3) s_1
= \hs_2 \hs_1 s_2 s_3 \hs_2 s_1 \,. 
\end{align*}
So
\begin{align*}
  s_3 \hs_2 s_1 \hs_2 s_3&   =   -(\hs_3 s_2 s_3)s_1 \hs_2 s_3
  + \hs_3 \hs_2 s_1 \hs_2 s_3  + (\hs_3 \hs_2 s_3) s_1 \hs_2 s_3\\
  & = - s_2 s_3 \hs_2 s_1 \hs_2 s_3 + \hs_3 \hs_2 s_1 \hs_2 s_3
  + s_2 \hs_3 \hs_2 s_1 \hs_2 s_3 \,.
\end{align*}
This proves Equation~\eqref{eq1}.
\begin{align*}
  (\hs_3 \hs_2) s_1 \hs_2 \hs_3& =  - (\hs_2 s_3 s_2)s_1 \hs_2 \hs_3
  + (s_2 s_3 \hs_2) s_1 \hs_2 \hs_3 + (s_2 s_3) s_1 \hs_2 \hs_3\\
  &\phantom{sii} -(s_2 \hs_3) s_1 \hs_2 \hs_3 - (\hs_2 s_3) s_1 \hs_2 \hs_3
  + (\hs_2 \hs_3) s_1 \hs_2 \hs_3\\
  &\phantom{sii} - (s_3 s_2) s_1 \hs_2 \hs_3 + (s_3 \hs_2) s_1 \hs_2 \hs_3
  + (\hs_3 s_2) s_1 \hs_2 \hs_3 \\
  &\phantom{sii} + (\hs_2 \hs_3 s_2) s_1 \hs_2 \hs_3 - (s_2 \hs_3 \hs_2)
  s_1 \hs_2 \hs_3 \,.
\end{align*}
Like in the previous case, most terms can be reduced to linear combinations
of Type 1 and Type 2 quantities: $\hs_2 s_3 s_2 s_1 \hs_2 \hs_3$, $s_2 s_3 s_1
\hs_2 \hs_3$, $s_2 \hs_3 s_1 \hs_2 \hs_3$, $\hs_2 s_3 s_1 \hs_2 \hs_3$, $\hs_2
\hs_3 s_1 \hs_2 \hs_3$, $s_3 s_2 s_1 \hs_2 \hs_3$, $\hs_3 s_2 s_1 \hs_2 \hs_3$ and
$\hs_2 \hs_3 s_2 s_1 \hs_2 \hs_3$. And we can write:
\begin{align*}
  \hs_3 \hs_2 s_1 \hs_2 \hs_3&   =   s_2 s_3 \hs_2 s_1 \hs_2 \hs_3
  + s_3 \hs_2 s_1 \hs_2 \hs_3 - s_2 \hs_3 \hs_2 s_1 \hs_2 \hs_3 \,.
\end{align*}
So Equation~\eqref{eq2} holds.  
\end{proof}

\begin{proposition}
\label{reductiontwofirst}
The words $s_{3} \hs_2 s_1 \hs_2 \hs_{3}$ and $\hs_{3} \hs_2 s_1 \hs_2 s_{3}$ reduce.
\end{proposition}

\begin{proof}
Let us start by writing an identity that will be useful subsequently. All
equalities here are to be understood modulo Type 1 and Type 2 terms. 
\be
\label{eq3}
\begin{aligned}
  s_2 s_3 \hs_2 s_1 (\hs_2 s_3)& \overset{\eqref{R2}}{=}
  - s_2 s_3 \hs_2 s_1 (\hs_2 s_3 s_2) + s_2 s_3 \hs_2 s_1 (s_2 s_3 \hs_2)
  + s_2 s_3 \hs_2 s_1 (s_2 s_3)\\
  &\phantom{spai} - s_2 s_3 \hs_2 s_1 (s_2 \hs_3) + s_2 s_3 \hs_2 s_1 (\hs_2 \hs_3)
  - s_2 s_3 \hs_2 s_1 (s_3 s_2) \\
  &\phantom{spai} + s_2 s_3 \hs_2 s_1 (s_3 \hs_2) + s_2 s_3 \hs_2 s_1 (\hs_3 s_2)
  - s_2 s_3 \hs_2 s_1 (\hs_3 \hs_2) \\
  &\phantom{spai} + s_2 s_3 \hs_2 s_1 (\hs_2 \hs_3 s_2)
  - s_2 s_3 \hs_2 s_1 (s_2 \hs_3 \hs_2) \,.
\end{aligned}
\ee
Some of the terms in the previous equality are linear combinations of Type 1 and
Type 2 terms:
\begin{align*}
  &s_2 s_3 (\hs_2 s_1 s_2) s_3 \hs_2    = s_2 s_3 s_1 s_2 \hs_1 s_3 \hs_2
  = s_2 s_1 (s_3 s_2 s_3) \hs_1 \hs_2 = s_2 s_1 s_2 s_3 s_2 \hs_1 \hs_2 \,,
\\
&s_2 s_3 (\hs_2 s_1 s_2) s_3    = s_2 s_3 s_1 s_2 \hs_1 s_3 = s_2 s_1 (s_3 s_2 s_3)
\hs_1 = s_2 s_1 s_2 s_3 s_2 \hs_1 \,,
\\
&s_2 s_3 (\hs_2 s_1 s_2) \hs_3    = s_2 s_3 s_1 s_2 \hs_1 \hs_3
= s_2 s_1 (s_3 s_2 \hs_3) \hs_1 = s_2 s_1 \hs_2 s_3 s_2 \hs_1 \,,
\\
&s_2 s_3 \hs_2 s_1 s_3 s_2    = s_2 (s_3 \hs_2 s_3 s_1) s_2
\overset{\eqref{R3}}{=} s_2 s_1 (s_3 \hs_2 s_3 s_2)
\overset{\text{reduction in }\la s_2,s_3\ra}{=} s_2 (s_1 \hs_3 s_2 \hs_3)  
\\
& \phantom{spacespaiii}    \overset{\eqref{R3}}{=} (s_2 \hs_3 s_2 \hs_3) s_1
\overset{\text{reduction in }\la s_2,s_3\ra}{=} \hs_3 s_2 \hs_3 s_1
\overset{\eqref{R3}}{=} s_1 \hs_3 s_2 \hs_3 \,,
\\
&s_2 s_3 \hs_2 s_1 s_3 \hs_2    = s_2 (s_3 \hs_2 s_3 s_1) \hs_2
\overset{\eqref{R3}}{=} s_2 s_1 (s_3 \hs_2 s_3 \hs_2)
\overset{\text{reduction in }\la s_2,s_3\ra}{=} s_2 (s_1 \hs_3 s_2 \hs_3)  
\\
&  \phantom{spacespaiii}  \overset{\eqref{R3}}{=} (s_2 \hs_3 s_2 \hs_3) s_1
\overset{\text{reduction in }\la s_2,s_3\ra}{=} \hs_3 s_2 \hs_3 s_1
\overset{\eqref{R3}}{=} s_1 \hs_3 s_2 \hs_3 \,,
\\
&s_2 s_3 \hs_2 s_1 \hs_3 s_2    = s_2 (s_3 \hs_2 \hs_3) s_1 s_2 =
s_2 \hs_2 \hs_3 s_2 s_1 s_2 \,,
\\
&s_2 s_3 \hs_2 s_1 \hs_3 \hs_2    = s_2 (s_3 \hs_2 \hs_3) s_1 \hs_2
= s_2 \hs_2 \hs_3 s_2 s_1 \hs_2 \,,
\\
&s_2 s_3 (\hs_2 s_1 s_2) \hs_3 \hs_2   = s_2 s_3 s_1 s_2 \hs_1 \hs_3 \hs_2
= s_2 s_1 (s_3 s_2 \hs_3) \hs_1 \hs_2 = s_2 s_1 \hs_2 s_3 s_2 \hs_1 \hs_2 \,.
\end{align*}
So Equation~\eqref{eq3} can be written in a simpler way.
\begin{align*}
  s_2 s_3 \hs_2 s_1 \hs_2 s_3&   =  - s_2 s_3 \hs_2 s_1 \hs_2 s_3 s_2
  + s_2 s_3 \hs_2 s_1 \hs_2 \hs_3 + s_2 s_3 \hs_2 s_1 \hs_2 \hs_3 s_2 \,.
\end{align*}
Let us reduce each of the three terms on the right hand side of the previous equality.
\begin{align*}
  (s_2 s_3 \hs_2) s_1 \hs_2 s_3 s_2   & = \hs_3 s_2 s_3 s_1 \hs_2 s_3 s_2
  = \hs_3 s_2 s_1 s_3 (\hs_2 s_3 s_2) = \hs_3 s_2 s_1 (s_3 s_3) s_2 \hs_3 \\
  & =  (t_0 + t_1 - 1) \, \hs_3 s_2 s_1 (s_3 s_2 \hs_3) + (t_0 + t_1 - t_0 t_1 )\,
  \hs_3 (s_2 s_1 s_2) \hs_3 - t_0 t_1 \, \hs_3 s_2 s_1 \hs_3 s_2 \hs_3 \\
  & =  (t_0 + t_1 - 1) \, \hs_3 (s_2 s_1 \hs_2) s_3 s_2 + (t_0 + t_1 - t_0 t_1 )\,
  \hs_3 s_1 s_2 s_1 \hs_3 - t_0 t_1 \, \hs_3 s_2 s_1 \hs_3 s_2 \hs_3 \\
  & =  (t_0 + t_1 - 1) \, \hs_3 \hs_1 s_2 s_1 s_3 s_2 + (t_0 + t_1 - t_0 t_1 )\,
  s_1 \hs_3 s_2 \hs_3 s_1 - t_0 t_1 \, \hs_3 s_2 s_1 \hs_3 s_2 \hs_3 \\
  & =  (t_0 + t_1 - 1) \, \hs_1 \hs_3 s_2 s_3 s_1 s_2 + (t_0 + t_1 - t_0 t_1 )\,
  s_1 s_1 \hs_3 s_2 \hs_3 - t_0 t_1 \, \hs_3 s_2 s_1 \hs_3 s_2 \hs_3 \\
  & =  (t_0 + t_1 - 1) \, \hs_1 s_2 s_3 \hs_2 s_1 s_2 + (t_0 + t_1 - t_0 t_1 )
  (s_1 s_1) \hs_3 s_2 \hs_3 - t_0 t_1 \, \hs_3 s_2 s_1 \hs_3 s_2 \hs_3 \\
   & = - t_0 t_1 \, \hs_3 s_2 s_1 \hs_3 s_2 \hs_3 \,.
\end{align*}
Now continuing down the rabbit hole, we can use Lemma~\ref{equivR2}'s version of
\eqref{R2} to write that:
\begin{align*}
  s_2 s_3 \hs_2 s_1 \hs_2 s_3 s_2   & =   - t_0 t_1 \, \hs_3 s_2 s_1 (\hs_3
  s_2 \hs_3) \overset{\eqref{R2}}{=} - (t_0 t_1)(t_0^{-1} t_1^{-1}) \, \hs_3
  s_2 s_1 (s_3 \hs_2 s_3) + \lambda \\
  & =  - (\hs_3 s_2 s_3) s_1 \hs_2 s_3 + \lambda = - s_2 s_3 \hs_2 s_1 \hs_2 s_3
  + \lambda \,,
\end{align*}
where $\lambda$ is a linear combination of the following words, all of which reduce:
\begin{align*}
  \hs_3 s_2 s_1 (\hs_2 \hs_3 \hs_2)  & =  \hs_3 (s_2 s_1 \hs_2) \hs_3 \hs_2
  = \hs_3 (\hs_1 s_2 s_1) \hs_3 \hs_2 
\\
&= \hs_1 (\hs_3 s_2 \hs_3 s_1) \hs_2 
\overset{\eqref{R3}}{=} (\hs_1 s_1) (\hs_3 s_2 \hs_3 \hs_2)
\overset{\text{reduction in }\la s_2,s_3\ra}{=} \hs_3 s_2 \hs_3 \,,
\\
\hs_3 s_2 s_1 (s_2 s_3 s_2)  & = \hs_3 (s_2 s_1 s_2) s_3 s_2
= \hs_3 s_1 s_2 s_1 s_3 s_2 = s_1 (\hs_3 s_2 s_3) s_1 s_2 = s_1 s_2 s_3
\hs_2 s_1 s_2 \,,
\\
\hs_3 s_2 s_1 (s_2 \hs_3 \hs_2)  & = \hs_3 (s_2 s_1 s_2) \hs_3 \hs_2
= \hs_3 s_1 s_2 s_1 \hs_3 \hs_2  
\\
&= s_1 (\hs_3 s_2 \hs_3 s_1) \hs_2  \overset{\eqref{R3}}{=} s_1 s_1
(\hs_3 s_2 \hs_3 \hs_2 ) \overset{\text{reduction in }\la s_2,s_3\ra}{=} s_1 s_1
(\hs_3 s_2 \hs_3) \,,
\\
\hs_3 s_2 s_1 (s_2 s_3 \hs_2)  & = \hs_3 (s_2 s_1 s_2) s_3 \hs_2
= \hs_3 (s_1 s_2 s_1) s_3 \hs_2 = s_1 (\hs_3 s_2 s_3) s_1 \hs_2
= s_1 s_2 s_3 \hs_2 s_1 \hs_2  \,,
\\
\hs_3 s_2 s_1 (\hs_3 \hs_2)  & = (\hs_3 s_2 \hs_3 s_1) \hs_2
\overset{\eqref{R3}}{=} s_1 (\hs_3 s_2 \hs_3 \hs_2)
\overset{\text{reduction in }\la s_2,s_3\ra}{=} s_1 \hs_3 s_2 \hs_3 \,,
\\
\hs_3 s_2 s_1 (\hs_3 s_2)  & = (\hs_3 s_2 \hs_3 s_1) s_2
\overset{\eqref{R3}}{=} s_1 (\hs_3 s_2 \hs_3 s_2)
\overset{\text{reduction in }\la s_2,s_3\ra}{=} s_1 \hs_3 s_2 \hs_3 \,,
\\
\hs_3 s_2 s_1 (s_3 \hs_2)  & = \hs_3 s_2 s_3 s_1 \hs_2
= (\hs_3 s_2 s_3) s_1 \hs_2 = s_2 s_3 \hs_2 s_1 \hs_2 \,,
\\
\hs_3 s_2 s_1 (s_3 s_2)  & = \hs_3 s_2 s_3 s_1 s_2
= (\hs_3 s_2 s_3) s_1 s_2 = s_2 s_3 \hs_2 s_1 s_2 \,,
\\
\hs_3 s_2 s_1 (\hs_2 \hs_3)  & = \hs_3 (s_2 s_1 \hs_2) \hs_3
= \hs_3 \hs_1 s_2 s_1 \hs_3 = \hs_1 (\hs_3 s_2 \hs_3 s_1)
\overset{\eqref{R3}}{=} \hs_1 s_1 \hs_3 s_2 \hs_3 = \hs_3 s_2 \hs_3 \,,
\\
\hs_3 s_2 s_1 (s_2 \hs_3)  & = \hs_3 (s_2 s_1 s_2) \hs_3 = \hs_3 s_1 s_2 s_1 \hs_3
= s_1 (\hs_3 s_2 \hs_3 s_1) \overset{\eqref{R3}}{=} (s_1 s_1) \hs_3 s_2 \hs_3 \,,
\\
\hs_3 s_2 s_1 (\hs_2 s_3)  & = \hs_3 (s_2 s_1 \hs_2) s_3 = \hs_3 \hs_1 s_2 s_1 s_3
= \hs_1 (\hs_3 s_2 s_3) s_1 = \hs_1 s_2 s_3 \hs_2 s_1 \,,
\\
\hs_3 s_2 s_1 (s_2 s_3)  & = \hs_3 (s_2 s_1 s_2) s_3 = \hs_3 s_1 s_2 s_1 s_3
= s_1 (\hs_3 s_2 s_3) s_1 = s_1 s_2 s_3 \hs_2 s_1 \,,
\\
\hs_3 s_2 s_1 (\hs_2)  & = \hs_3 s_2 s_1 \hs_2 \,,
\\
\hs_3 s_2 s_1 (s_2)  & = \hs_3 s_2 s_1 s_2 \,.
\end{align*}
This means that modulo Type 1 and Type 2 terms:
$$
s_2 s_3 \hs_2 s_1 \hs_2 s_3 s_2    =  - s_2 s_3 \hs_2 s_1 \hs_2 s_3  \,,
$$
and Equation~\eqref{eq3} can be further simplified:
$$
\cancel{s_2 s_3 \hs_2 s_1 \hs_2 s_3}   =  \cancel{- s_2 s_3
  \hs_2 s_1 \hs_2 s_3 s_2} + s_2 s_3 \hs_2 s_1 \hs_2 \hs_3
+ s_2 s_3 \hs_2 s_1 \hs_2 \hs_3 s_2 \,.
$$
Also:
\begin{align*}
  s_2 s_3 \hs_2 s_1 (\hs_2 \hs_3 s_2)&\phantom{i}   = s_2 s_3 \hs_2 s_1 s_3
  \hs_2 \hs_3 = (s_2 s_3 \hs_2) s_3 s_1 \hs_2 \hs_3 = \hs_3 s_2 (s_3 s_3) s_1
  \hs_2 \hs_3\\
  & \overset{\eqref{R1}}{=} (t_0 + t_1 - 1) \, (\hs_3 s_2 s_3) s_1 \hs_2 \hs_3
  + (t_0 + t_1 - t_0 t_1 )\, \hs_3 (s_2 s_1 \hs_2) \hs_3 - t_0 t_1 \,
  \hs_3 s_2 \hs_3 s_1 \hs_2 \hs_3\\
  &\phantom{i}  = (t_0 + t_1 - 1) \, s_2 s_3 \hs_2 s_1 \hs_2 \hs_3 + (t_0 + t_1
  - t_0 t_1 )\, \hs_3 \hs_1 s_2 s_1 \hs_3 - t_0 t_1 \, \hs_3 s_2 s_1 (\hs_3
  \hs_2 \hs_3) \\
  & \phantom{i}  = (t_0 + t_1 - 1) \, s_2 s_3 \hs_2 s_1 \hs_2 \hs_3
  + (t_0 + t_1 - t_0 t_1 )\, \hs_1 (\hs_3 s_2 \hs_3 s_1) - t_0 t_1 \,
  \hs_3 (s_2 s_1 \hs_2) \hs_3 \hs_2 \\
  & \overset{\eqref{R3}}{=} (t_0 + t_1 - 1) \, s_2 s_3 \hs_2 s_1 \hs_2 \hs_3
  + (t_0 + t_1 - t_0 t_1 )\, (\hs_1 s_1) \hs_3 s_2 \hs_3 - t_0 t_1 \, \hs_3
  \hs_1 s_2 s_1 \hs_3 \hs_2 \\
  & \phantom{i}  = (t_0 + t_1 - 1) \, s_2 s_3 \hs_2 s_1 \hs_2 \hs_3
  + (t_0 + t_1 - t_0 t_1 )\, \hs_3 s_2 \hs_3 - t_0 t_1 \, \hs_1
  (\hs_3 s_2 \hs_3 s_1) \hs_2 \\
  & \overset{\eqref{R3}}{=} (t_0 + t_1 - 1) \, s_2 s_3 \hs_2 s_1 \hs_2
  \hs_3 - t_0 t_1 \, \hs_1 s_1 (\hs_3 s_2 \hs_3 \hs_2) \\
  & \phantom{i}  = (t_0 + t_1 - 1) \, s_2 s_3 \hs_2 s_1 \hs_2 \hs_3
  - t_0 t_1 \, \hs_3 s_2 \hs_3 \hs_2 \\
  & \phantom{i}  = (t_0 + t_1 - 1) \, s_2 s_3 \hs_2 s_1 \hs_2 \hs_3
  \text{ (reduction in $\la s_2,s_3\ra$)}\,.
\end{align*}
So Equation~\eqref{eq3} can now be written:
$$
0 = s_2 s_3 \hs_2 s_1 \hs_2 \hs_3 + (t_0 + t_1 - 1) \, s_2 s_3 \hs_2 s_1 \hs_2
\hs_3 \text{ , i.e. } 0 = (t_0 + t_1) \, s_2 s_3 \hs_2 s_1 \hs_2 \hs_3 \,.
$$
So $s_2 s_3 \hs_2 s_1 \hs_2 \hs_3$ reduces. And multiplying by $\hs_2$
from the left, we deduce that $s_3 \hs_2 s_1 \hs_2 \hs_3$ reduces as well. Indeed:
\begin{align*}
  s_2 s_3 \hs_2 s_1 \hs_2 \hs_3 & = a(s_1,s_2) + b(s_1,s_2) \, s_3 \,
  c(s_1,s_2) + d(s_1,s_2) \, \hs_3 \, e(s_1,s_2) \\
  &\phantom{sii} + \lambda \, \hs_3 s_2 \hs_3 + \mu \, s_1 \hs_3 s_2 \hs_3
  + \nu \, \hs_1 \hs_3 s_2 \hs_3 \,,
\end{align*}
so
\begin{align*}
  s_3 \hs_2 s_1 \hs_2 \hs_3 & = \hs_2 \, a(s_1,s_2) + \hs_2 \, b(s_1,s_2) \, s_3
  \, c(s_1,s_2) + \hs_2 \, d(s_1,s_2) \, \hs_3 \, e(s_1,s_2) \\
  &\phantom{sii} + \lambda \, \hs_2 \hs_3 s_2 \hs_3 + \mu \, \hs_2 s_1
  \hs_3 s_2 \hs_3 + \nu \, \hs_2 \hs_1 \hs_3 s_2 \hs_3 \\
  & =  \tilde{a}(s_1,s_2) +  \tilde{b}(s_1,s_2) \, s_3 \, c(s_1,s_2)
  + \tilde{d}(s_1,s_2) \, \hs_3 \, e(s_1,s_2) \\
  & \phantom{sii} + \lambda \, \hs_3 s_2 \hs_3 + \mu \, s_1 (\hs_2 \hs_3 s_2 \hs_3)
  + \nu \, \hs_2 (\hs_1 \hs_3 s_2 \hs_3) \text{ (reduction in $\la s_2,s_3\ra$)}\\
  &=  \tilde{a}(s_1,s_2) +  \tilde{b}(s_1,s_2) \, s_3 \, c(s_1,s_2)
  + \tilde{d}(s_1,s_2) \, \hs_3 \, e(s_1,s_2) \\
  & \phantom{sii} + \lambda \, \hs_3 s_2 \hs_3 + \mu \, s_1 \hs_3 s_2 \hs_3
  + \nu \, (\hs_2 \hs_3 s_2 \hs_3) \hs_1 \text{ (reduction in
$\la s_2,s_3\ra$ + \eqref{R3})} \\
  &=  \tilde{a}(s_1,s_2) +  \tilde{b}(s_1,s_2) \, s_3 \, c(s_1,s_2)
  + \tilde{d}(s_1,s_2) \, \hs_3 \, e(s_1,s_2) \\
  & \phantom{sii} + \lambda \, \hs_3 s_2 \hs_3 + \mu \, s_1 \hs_3 s_2 \hs_3
  + \nu \, \hs_3 s_2 \hs_3 \hs_1 \text{ (reduction in $\la s_2,s_3\ra$)} \\
  &=  \tilde{a}(s_1,s_2) +  \tilde{b}(s_1,s_2) \, s_3 \, c(s_1,s_2)
  + \tilde{d}(s_1,s_2) \, \hs_3 \, e(s_1,s_2) \\
  & \phantom{sii} + \lambda \, \hs_3 s_2 \hs_3 + \mu \, s_1 \hs_3 s_2 \hs_3
  + \nu \, \hs_1 \hs_3 s_2 \hs_3 \text{ (\eqref{R3})} \\
&= 0 \,.
\end{align*}

Since $s_3 \hs_2 s_1 \hs_2 \hs_3$ reduces, if we start the previous computation
again writing everything from right to left when it was written from left to
right, we find in the same way that $\hs_3 \hs_2 s_1 \hs_2 s_3$ reduces.
\end{proof}

\begin{proposition}
The words $s_{3} \hs_2 s_1 \hs_2 s_{3}$ and $\hs_{3} \hs_2 s_1 \hs_2 \hs_{3}$ reduce.
\end{proposition}

\begin{proof}
Using Proposition~\ref{reductiontwofirst}, Equations~\eqref{eq1} and~\eqref{eq2}
can be written in a simpler way: 
$$
{\eqref{eq1} : }\,s_3 \hs_2 s_1 \hs_2 s_3 + s_2 s_3 \hs_2 s_1 \hs_2 s_3 = 0 \, ,
$$
$$
{\eqref{eq2} : }\,\hs_3 \hs_2 s_1 \hs_2\hs_3 + s_2 \hs_3 \hs_2 s_1 \hs_2 \hs_3 = 0 \,.
$$

To prove e.g. that $s_3 \hs_2 s_1 \hs_2 s_3$ reduces, one way is to find a
different equation involving $s_3 \hs_2 s_1 \hs_2 s_3$ and $s_2 s_3 \hs_2
s_1 \hs_2 s_3$. To do so, we write down $\eqref{R3} \cdot \hs_2 s_3$ explicitly
and we find that modulo Type 1 and Type 2 terms, it can be expressed as follows:
\be
\label{eq4}
a_{78} \, (s_2 s_3 \hs_2 s_1) \hs_2 s_3 + (a_{39} - a_{75} - a_{53}
- a_{47}) \, s_3 \hs_2 s_1 \hs_2 s_3 = 0 \,.
\ee

Therefore the system
$$
(\Sigma): \left\{
\begin{array}{l}
s_3 \hs_2 s_1 \hs_2 s_3 + s_2 s_3 \hs_2 s_1 \hs_2 s_3 = 0 \\
a_{78} \, (s_2 s_3 \hs_2 s_1) \hs_2 s_3 + (a_{39} - a_{75} - a_{53} - a_{47}) \,
s_3 \hs_2 s_1 \hs_2 s_3 = 0
\end{array}
\right.
$$
has the following determinant:
\begin{align*}
   \text{det}(\Sigma)&\phantom{i}   =\left|
\begin{array}{cc}
a_{39} - a_{75} - a_{53} - a_{47}&1\\
a_{78}&1
\end{array}
\right| = - a_{78} + a_{39} - a_{75} - a_{53} - a_{47}\\
& \phantom{i}  = \frac{(t_0 + t_1 -1)(1 + t_0 t_1 + t_0^2 t_1
  + t_0 t_1^2)}{(t_0 + t_1)(t_0 t_1 +1)(t_0 t_1 -1)} \neq 0 \,.
\end{align*}
So $(\Sigma)$ is an invertible system. Thus $s_3 \hs_2 s_1 \hs_2 s_3$ (and
$s_2 s_3 \hs_2 s_1 \hs_2 s_3$) can be reduced.

Similarly, we can prove that $\hs_3 \hs_2 s_1 \hs_2\hs_3$ (and $s_2 \hs_3
\hs_2 s_1 \hs_2 \hs_3$) reduce by considering the system comprised of the
reduced version of Equation~\eqref{eq2} and $\eqref{R3} \cdot \hs_2 \hs_3$.
\end{proof}

Summing up, Lemma~\ref{fourstrand} is now proved.


\bibliographystyle{hamsalpha}
\bibliography{biblio}
\end{document}